\documentclass[12pt,twoside]{article}

\usepackage{amsmath}
\usepackage{amssymb}
\usepackage{amscd}
\usepackage{bbm}
\usepackage{latexsym}
\usepackage[all]{xy}
\usepackage{makeidx}
\usepackage{tocloft}
\usepackage{fancyhdr}

\usepackage{ifthen}

\newcommand{\lsection}[2][""]{%
    \ifthenelse{\equal{#1}{""}}{%
        \section{#2}
    }{%
        \renewcommand{\sectionmark}[1]{\markright{\thesection.\ \MakeUppercase{#1}}}
        \section{#2}
        \renewcommand{\sectionmark}[1]{\markright{\thesection.\ \MakeUppercase{##1}}}
    }
}

\newcommand{\lchapter}[2][""]{%
    \ifthenelse{\equal{#1}{""}}{%
        \chapter{#2}
    }{%
        \renewcommand{\chaptermark}[1]{\markboth{\MakeUppercase{\chaptername\ \thechapter.\ #1}}{}}
        \chapter{#2}
        \renewcommand{\chaptermark}[1]{\markboth{\MakeUppercase{\chaptername\ \thechapter.\ ##1}}{}}
    }
}

\def\a{\alpha}

\def\z{\zeta}

\def\iso{isomorphism}
\def\homo{homomorphism}

\def\vb{vector bundle}

\def\mfd{manifold}

\def\nbd{neighborhood}

\def\lra{\longrightarrow}
\def\F{{\cal F}}
\def\G{{\cal G}}
\def\E{{\cal E}}

\def\A{{\cal A}}
\def\U{{\cal U}}
\def\V{{\cal V}}
\def\Na{\mbox{\boldmath $\nabla$}}

\def\na{\nabla}
\def\o{\omega}

\def\t{\theta}

\def\vv {\vskip.2cm}
\def\C{{\mathbb C}}
\def\R{{\mathbb R}}

\def\O{{\cal O}}

\def\P{\mathbb P}
\def\bo1{{\text{\bold 1}}}

\def\Randmarke#1{\vadjust{\vbox to 0pt
                {\vss \hbox to \hsize{\hskip\hsize
                                     \quad #1\hss}\vskip3.5pt}}}
\def\Randpfeilmarke#1{\Randmarke{$\scriptscriptstyle\Leftarrow$#1}}
\def\bi#1 { $^{\bf!}$ {\{\sl #1 \}} \Randpfeilmarke{\bf !!}}

\newtheorem{theorem}{Theorem}[section]
\newtheorem{lemma}[theorem]{Lemma}
\newtheorem{corollary}[theorem]{Corollary}
\newtheorem{proposition}[theorem]{Proposition}

\newtheorem{exa}[theorem]{Example}
\newenvironment{example}{\begin{exa} \em}{\end{exa}}
\newtheorem{exas}[theorem]{Examples}

\newtheorem{prope}[theorem]{Property}

\newtheorem{defini}[theorem]{Definition}
\newenvironment{definition}{\begin{defini} \em}{\end{defini}}

\newtheorem{rema}[theorem]{Remark}
\newenvironment{remark}{\begin{rema} \em}{\end{rema}}

\newenvironment{equationth}{\stepcounter{theorem}\begin{equation}}{\end{equation}}

\newenvironment{proof}{{\noindent \sc Proof : } }{\mbox{ }\hfill$\Box$
                        \vspace{1.5ex} \par}

\def\hfleche#1#2{\smash{\mathop{\vbox{\hbox to 8.5mm{{#1}}}}\limits^{#2}}}

\newcommand{\Hom}{\mathop\textrm{Hom}\nolimits}

\title {\bf  \Large {Localization of Atiyah classes}}
\author{M. Abate, F. Bracci, T. Suwa and F. Tovena}

\begin{document}

\pagestyle{plain}


\maketitle

In \cite{At}, M. Atiyah introduced the notion of complex analytic connections and constructed characteristic classes of holomorphic vector bundles, in the Dolbeault cohomology, via the  obstruction to the existence of such a connection for the bundle.
In this paper,  we study localization problems of those classes, which we will call {\it Atiyah classes}.

We first reconstruct the classes using $C^\infty$ connections of type $(1,0)$ and developing a Chern-Weil type theory for these connections. In fact, the use of this type of connections  was already present in \cite{At} in the framework of principal bundles and the idea of incorporating this into the Chern-Weil theory had  been noted in \cite{BB1}. We further exploit
this approach. If we treat the Atiyah classes this way, we may represent a class by a $\bar\partial$-closed form, which is a part of the corresponding Chern form. Combined with  the \v Cech-Dolbeault cohomology, this viewpoint is particularly suited for localization problems.


In the case of Chern classes, localization problems arise naturally in many contexts. For example, taking the Poincar\'e-Hopf index theorem as a model case, a number of works had been done for the indices and residues of vector fields. Then a  general  residue theory for singular holomorphic foliations was developed by P. Baum and R. Bott  in \cite{BB2}.  This can be interpreted as a localization theory of the characteristic classes of the normal sheaf of the foliation based on the Bott vanishing theorem. The point here is that, once we have a certain  vanishing theorem, we have a localization theory. The index theorem of C.~Camacho and P. Sad in \cite{CS}, which was effectively used in their  proof of the existence of separatrices for holomorphic vector fields on the complex plane, was also interpreted and generalized in this context in \cite{LS}. There are some other residues for singular foliations that can be captured from this viewpoint and they
 are systematically treated in  \cite{Su2}. Here the combination of
the Chern-Weil theory and the \v Cech-de~Rham cohomology, originally due to \cite{Leh2}, is a very convenient tool to described the theory.

The philosophy and the techniques  above turned out to be very effective in other problems related to characteristic classes. For example, there is a work \cite{Ab}  in the  complex dynamical systems that corresponds to the aforementioned work of Camacho-Sad. The index theorem used there, which was originally proved in a different way, is shown to be proved,  in \cite{BTo}, in the framework of residue theory of singular foliations as above. This viewpoint gave unification of index and residue theories both for foliations and mappings and further generalizations (\cite{ABT0, ABT}).

They are also applied to the study of characteristic classes of singular varieties, that are summarized in  \cite{BSS}. The residue theory in this framework led to an analytic intersection theory on singular varieties \cite{Su7}, with an application in the discrete dynamical system on singular surfaces \cite{BSu1}. See \cite{Pe} for another development in this direction.

In this paper we try to make analogous studies in the case of Atiyah classes, compare with and complement the theories in the case of Chern classes.

In Section 1, we describe the Atiyah classes using connections of type $(1,0)$, as mentioned above, and compare these with the original definition in Section 2. Then we recall the \v Cech-Dolbeault cohomology theory in Section 3. In Section 4, we define Atiyah classes in the \v Cech-Dolbeault cohomology and explain the basic principle of localization. Each time we have a ``vanishing theorem", we have a corresponding localization theory and the associated residues. In Section 5, we briefly discuss the localization by sections, or more generally, by frames.
In Section 6, we prove a Bott type vanishing theorem for non-singular distributions, which lay foundation to the residue theory for singular distributions. As an important example, we give  the vanishing theorem coming from the ``Camacho-Sad action" in Section 7. In Section 8, we discuss singular holomorphic distributions and in Section 9, we give an example of Atiyah residues for some singular distribution.


\lsection{Atiyah classes}\label{atiyah}

For details of the Chern-Weil theory of characteristic classes of complex \vb s, we refer to  \cite{BB2}, \cite{Bo1}, \cite{MS}, \cite{Su2}. Here we use the notation in \cite{Su2} (with connection and curvature matrices transposed and $r$ and $\ell$  interchanged).

\subsection{Atiyah forms}

Let $M$ be a complex \mfd\ and  $E$  a holomorphic \vb\   over $M$ of rank $\ell$. For an open set $U$ in  $M$, we denote by
$A^p(U)$ the complex vector space of complex valued $C^{\infty}$ $p$-forms on $U$. Also,
we
let $A^p(U,E)$ be the vector space of ``$E$-valued $p$-forms" on $U$, {\it i.e.}, $C^{\infty}$ sections of the bundle
$\bigwedge^p{(T_{\R}^cM)}^*\otimes E$ on $U$, where ${(T_{\R}^cM)}^*$ denotes
the dual of the complexification of the real tangent bundle $T_{\R}M$ of $M$.
Thus $A^0(U)$ is the ring of $C^{\infty}$ functions and
$A^0(U,E)$ is the $A^0(U)$-module of $C^{\infty}$ sections of $E$ on $U$.

\begin{definition} A ($C^\infty$) {\it connection} for $E$  is a ${\C}$-linear map
\[
\na\colon A^0(M,E)\lra A^1(M,E)
\]
satisfying the Leibniz rule
$$
\na(fs)=df\otimes s+f\na(s)\ \ \ \ \text{for}\ \ \ f\in
A^0(M)\ \text{and}\ s\in A^0(M,E).
$$
\end{definition}

\begin{definition} For $r=1,\dots,\ell$, an \emph{$r$-frame} of $E$ on an open set $U$ is a collection $s^{(r)}=(s_1,\dots,s_r)$ of $r$ sections  of $E$ linearly independent everywhere on $U$. An $\ell$-frame is simply called a \emph{frame.}
\end{definition}

\begin{definition}\label{triviality}
 Let $\na$ be a connection  for $E$ on $U$, and $s^{(r)}=(s_1,\dots,s_r)$
 an $r$-frame of~$E$. We say that $\nabla$ is \emph{$s^{(r)}$-trivial} if  $\na(s_i)=O$ for~$i=1,\ldots,r$.
\end{definition}

A connection $\nabla$ for $ E$ induces a ${\C}$-linear map
\[
\na\colon A^1(M,E)\lra A^2(M,E)
\]
satisfying
\[
\na(\o\otimes s)=d\o\otimes s-\o\wedge \na(s)\ \ \ \ \hbox{for}\ \ \ \o\in A^1(M)\
\hbox{and}\ s\in A^0(M,E).
\]

The composition
\[
K=\na\circ \na\colon A^0(M,E)\lra A^2(M,E)
\]
is called the {\it curvature} of $\na$.
It is not difficult to see that $K$ is $A^0(M)$-linear; hence it can be thought of as a $C^\infty$ $2$-form with coefficients in the bundle $\Hom(E,E)$.

Notice that a~connection is a~local operator, {\it i.e.}, it is also defined on local sections. This fact
allows us to obtain local representations of a connection
and its curvature by matrices whose entries are differential forms.
Thus suppose that $\nabla$~is a~connection for $E$.
If $e^{(\ell)} = (e_1, \dots, e_{\ell})$ is a~frame of~$E$ on~$U$,
we may write, for $i = 1, \dots, \ell$,
\[
  \nabla(e_i) = \sum_{j=1}^{\ell} \theta^j_{i} \otimes e_j
  \qquad\textrm{with }\ \theta^j_{i}\ \ \textrm{in}\ \ A^1(U).
\]
The matrix $\theta = (\theta^j_{j})$ is the \emph{connection matrix} of~$\nabla$ with respect to~$e^{(\ell)}$.
Also, from the definition we get
\[
  K(e_i) = \sum_{j=1}^{\ell} \kappa^j_{i} \otimes e_j
  \qquad\textrm{with }\ \
    \kappa^j_{i}
  = d\theta^j_{i} + \sum_{k=1}^{\ell}\theta^j_{k}\wedge \theta^k_{i} .
\]
We call $\kappa = (\kappa^j_{i})$ the \emph{curvature matrix} of~$\nabla$
with respect to~$e^{(\ell)}$.
If $\tilde e^{(\ell)} = (\tilde e_1, \dots, \tilde e_{\ell})$ is another frame of~$E$ on~$\tilde U$,
we have $\tilde e_i = \sum_{j=1}^{\ell} a^j_{i} e_j$
for suitable functions $a^j_i\in C^{\infty}(U\cap\tilde U)$, and
the matrix $A = (a^j_{i})$ is non-singular at each point of $U \cap \tilde U$.
If we denote by~$\tilde\theta$ and~$\tilde\kappa$ the connection and curvature matrices
of~$\nabla$ with respect to~$\tilde e^{(\ell)}$ we have
\begin{equationth}\label{framechange}
  \tilde\theta = A^{-1} \cdot dA + A^{-1} \theta A
  \quad\text{and}\quad
  \tilde\kappa = A^{-1} \kappa A
  \quad \text{in}\quad
  U \cap\tilde U.
\end{equationth}%

Up to now $E$ could have been only a $C^\infty$ complex \vb. Now we use the assumption that $E$ is holomorphic.

\begin{definition} A connection $\na$ for $E$ is \emph{of type $(1,0)$} (or a \emph{$(1,0)$-connection\/}) if the entries of the connection matrix with respect to a holomorphic frame are forms of type $(1,0)$.
\end{definition}

\begin{remark}  (1)  It is easy to check that the above property does not depend on the choice of the
holomorphic frame.

\vv

\noindent
(2) A holomorphic \vb\ always admits a $(1,0)$-connection. In fact let $\V=\{V_\lambda\}$ be an open covering of $M$ trivializing $E$. For each $\lambda$, let $\na_\lambda$ be the connection trivial with respect to some holomorphic frame on $V_\lambda$. If we take a partition of unity $\{\rho_\lambda\}$ subordinate to $\V$ and set $\na=\sum_\lambda\rho_\lambda\na_\lambda$, then $\nabla$ is a $(1,0)$-connection for $E$.
\end{remark}

If $\na$ is a $(1,0)$-connection for $E$, we may write its curvature $K$ as
\[
K=K^{2,0}+K^{1,1}
\]
with $K^{2,0}$ and $K^{1,1}$, respectively, a $(2,0)$-form and a $(1,1)$-form with coefficients in $\Hom(E,E)$. Locally, if $\theta$ and $\kappa$ are respectively the connection and the curvature
matrices of $\nabla$ with respect to a (local) holomorphic frame of $E$, then
we can decompose $\kappa=\kappa^{2,0}+\kappa^{1,1}$ according to type, and $K^{2,0}$
and $K^{1,1}$ are respectively represented by
\begin{equationth}\label{curvform}
\kappa^{2,0}=\partial\theta+\theta\wedge\theta\qquad\text{and}\qquad
\kappa^{1,1}=\bar\partial\theta\;.
\end{equationth}%
Thus $K^{1,1}$, being locally $\bar\partial$-exact, is a $\bar\partial$-closed $(1,1)$-form with coefficients in $\Hom(E,E)$.

With respect to another holomorphic frame, $K^{1,1}$ is represented by a  matrix similar to $\kappa^{1,1}$ (cf. (\ref{framechange})). Thus for each elementary symmetric polynomial $\sigma_p$ (with $p=1,2,\dots$)
we may define a $\bar\partial$-closed $C^{\infty}$
$(p,p)$-form $\sigma_p(K^{1,1})$ on $M$. Locally it is given by  $\sigma_p(\kappa^{1,1})$, which is the coefficient of $t^p$ in the expansion
\[
\det(I+t\kappa^{1,1})=1+\sigma_1(\kappa^{1,1})t+\cdots+\sigma_p(\kappa^{1,1})t^p+\cdots.
\]
In particular, $\sigma_1(\kappa^{1,1})={\rm tr}(\kappa^{1,1})$ and $\sigma_\ell(\kappa^{1,1})=\det(\kappa^{1,1})$.

\begin{definition} Let $\nabla$ be a $(1,0)$-connection for a holomorphic vector bundle~$E$
of rank~$\ell$. For $p=1,\ldots,\ell$, we define
the \emph{$p$-th Atiyah form}~$a^p(\nabla)$ of~$\nabla$ by
\[
a^p(\na)=\left(\frac{\sqrt{-1}}{2\pi}\right)^p\sigma_p(K^{1,1}).
\]
It is a $\bar\partial$-closed $(p,p)$-form on $M$.

More generally, if $\varphi$ is a symmetric homogeneous polynomial of degree $d$, we may write
$\varphi=P(\sigma_1,\sigma_2,\dots)$ for a suitable polynomial $P$. Then we define
\emph{the Atiyah form~$\varphi^A(\nabla)$ of $\na$ associated to $\varphi$} by
\[
\varphi^A(\na)=P(a^1(\na),a^2(\na),\ldots);
\]
it is a $\bar\partial$-closed $(d,d)$-form on $M$.
\end{definition}

\begin{remark}  The $p$-th Chern form $c^p(\na)$ of $\na$ is defined by
\[
c^p(\na)=\left(\frac{\sqrt{-1}}{2\pi}\right)^p\sigma_p(\kappa),
\]
which is a closed $(2p)$-form having components of bidegrees $(2p,0),\ldots, (p,p)$. The Atiyah form $a^p(\na)$ is then the $(p,p)$-component of $c^p(\na)$. In particular, $a^n(\na)=c^n(\na)$, where $n$ denotes the dimension of $M$.

More generally, the Atiyah form $\varphi^A(\na)$ of $\na$ associated to a symmetric homogeneous polynomial~$\varphi$ of degree~$d$ is the component of type $(d,d)$ of the Chern form
$\varphi(\na)=P(c^1(\na),c^2(\na),\dots)$ associated to~$\varphi$. Again, if $d=n$ then
$\varphi^A(\nabla)=\varphi(\nabla)$.
\end{remark}

\subsection{Atiyah classes}\label{subA}

Let $E$ be a holomorphic \vb\ over a complex \mfd\ $M$.
As in the case of Chern forms, to any set of at least two $(1,0)$-connections for~$E$
one can associate a \emph{difference form} starting from the usual Bott difference
forms.
Here we recall the construction in the case of two connections and refer to \cite[Proposition 5.4]{Su8} for the general case.

Thus, given two $(1,0)$-connections $\na_0$ and $\na_1$ for $ E$,
we consider the vector bundle
$E\times \mathbb R\to  M\times \mathbb R$ and define the connection $\tilde\na$  for it  by
$\tilde\na=(1-t)\na_0+t\na_1$,
where $t$ denotes
a coordinate on $\mathbb R$. Denoting by $[0,1]$ the unit interval
and by $\pi:M\times [0,1]\to M$ the projection, we have the integration along the fiber
\[
\pi_*:A^{2p}(M\times[0,1])\lra A^{2p-1}(M).
\]
Then we set
\[
c^p(\na_0,\na_1)=\pi_*(c^p(\tilde\na)).
\]

The Atiyah difference form $a^p(\na_0,\na_1)$ is the $(p,p-1)$-component of $c^p(\na_0,\na_1)$.
It is alternating in $\na_0$ and $\na_1$ and satisfies
\begin{equationth}\label{difference}
\bar\partial a^p(\na_0,\na_1)=a^p(\na_1)-a^p(\na_0).
\end{equationth}%

This in particular shows that, if $\na$ is a $(1,0)$-connection for $E$, the class of $a^p(\na)$ in $H^{p,p}_{\bar\partial}(M)$ does not depend on the choice of $\na$.

Similarly, if $\varphi$ is a symmetric homogeneous
polynomial of degree~$d$ then there is a $(d,d-1)$-form $\varphi^A(\na_0,\na_1)$ alternating
in~$\nabla_0$ and $\nabla_1$ and satisfying
\[
\varphi^A(\na_1)-\varphi^A(\na_0)=\bar\partial \varphi^A(\na_0,\na_1).
\]

Then we can introduce the following definition\,:

\begin{definition} Let $E$ be a holomorphic vector bundle~$E$
of rank~$\ell$. For $p=1,\ldots,\ell$ the \emph{$p$-th Atiyah class} $a^p(E)$ of $E$  is the class
represented by $a^p(\na)$ in $H^{p,p}_{\bar\partial}(M)$, where $\na$ is an arbitrary $(1,0)$-connection for $ E$.

Similarly, if $\varphi$ is a symmetric homogeneous polynomial of degree $d$,
the \emph{Atiyah class $\varphi^A(E)$ of $E$ associated to $\varphi$}
is the class  of $\varphi^A(\nabla)$ in $H^{d,d}_{\bar\partial}(M)$, where $\nabla$ is an arbitrary $(1,0)$-connection for $ E$.
\end{definition}

\begin{remark}\label{top}  If $n$ denotes the dimension of $M$, there is a canonical surjective map $H^{n,n}_{\bar\partial}(M)\lra H^{2n}_d(M)$, which assigns the class of a form $\o$ to the class of $\o$. If $d=n$, then $\varphi^A(\nabla)=\varphi(\nabla)$ for any
$(1,0)$-connection $\na$ for $ E$ and the image of $\varphi^A(E)$ by the above map is $\varphi(E)$. In particular, if $M$ is compact, $\int_M\varphi^A(E)=\int_M\varphi(E)$.

 Moreover, if $d=n$,  then $\varphi^A(\nabla_0,\nabla_1)$ also coincides
with the usual Bott difference form~$\varphi(\nabla_0,\nabla_1)$ for any pair of
$(1,0)$-connections $\na_0$,~$\nabla_1$ for $ E$.
\end{remark}

\subsection{Atiyah classes on compact K\"ahler manifolds}

Let $M$ be complex manifold (not necessarily K\"ahler) and  $E$
a holomorphic vector bundle on $M$. Let $h$ be any Hermitian
metric on $E$ and let $\nabla^h$ be the associated metric
connection, {\it i.e.}, $\nabla^h$ is the unique $(1,0)$-connection compatible with $h$. The curvature $K$ of $\nabla$ is then
of type $(1,1)$, and hence
\[
c^p(\nabla^h)=a^p(\nabla^h) \qquad \textrm{for all }p\geq 1.
\]
In other words, Atiyah and Chern classes of the same degree can be represented by
the same form. Of course, as classes they are different, because they belong to
two different cohomology groups\,:
$c_p(E)=[c^p(\nabla^h)]\in H^{2p}_{d}(M)$, the de~Rham cohomology of $M$,  while
$a^p(E)=[a^p(\nabla^h)]\in H^{p,p}_{\bar\partial}(M)$, the Dolbeault cohomology of~$M$.

However, if $M$ is compact K\"ahler, the Hodge decomposition yields a canonical
injection $H^{p,p}_{\bar\partial}(M)\hookrightarrow H^{2p}_d(M)$, and hence
we obtain the following useful relation\,:

\begin{proposition}\label{compactKAtiyah}
Let $M$ be a compact K\"ahler manifold and  $E$ a
holomorphic vector bundle on $M$. Let
$H\colon H^{p,p}_{\bar\partial}(M)\to H^{2p}_d(M)$ be the injection
given by the Hodge decomposition. Then
\[
H\bigl(a^p(E)\bigr)=c^p(E)\qquad\textrm{for all }\ \ p\geq 1.
\]
\end{proposition}

\lsection{Atiyah classes via complex analytic connections}\label{atiyahold}

Atiyah classes were originally introduced by Atiyah in \cite{At}, with a different construction.
In this section we show that our definition yields the same classes.

Let $M$ be a complex \mfd\ and $\O$ the sheaf of germs of holomorphic functions on $M$. For a holomorphic \vb\ $E$  over $M$ we denote by
$\E={\cal O}(E)$ the sheaf of germs of holomorphic sections of $E$. We also denote by $\Theta={\cal O}(TM)$ and $\Omega^1={\cal O}(T^*M)$ the sheaves of germs holomorphic vector fields and of $1$-forms on $M$. All tensor products
in this section will be over the sheaf~$\O$.

\begin{definition}\label{caconnection}
A \emph{holomorphic} (or \emph{complex analytic\/}) {\it  connection} for $ E$ is a  \homo\ of
sheaves of $\C$-vector spaces
\[
\mbox{\boldmath $\nabla$}\colon\E\longrightarrow \Omega^1\otimes\E
\]
satisfying
\[
\Na(fs)=df\otimes s+f\Na(s)\quad\text{for}\quad f\in{\cal O}\  \text{and}\  s\in\E.
\]

If $e^{(r)}=(e_1,\ldots,e_r)$ is a local holomorphic $r$-frame of~$E$, we shall say that
$\Na$ is \emph{$e^{(r)}$-trivial} if $\Na e_j\equiv O$ for $j=1,\ldots,r$.
\end{definition}

\begin{remark}\label{holoCinfty}
A holomorphic connection $\Na$ on a holomorphic vector bundle $E$ induces naturally
a $(1,0)$-connection $\na$. In fact, let $s$ be a $C^\infty$ section of $E$. Let $U$ be an open set trivializing $E$ and let $(e_1,\dots,e_\ell)$ be a holomorphic frame on $U$. Write $s=\sum_{i=1}^\ell f^i e_i$ for suitable $f^i\in C^\infty(U)$, and set $\na s=\sum_{i=1}^\ell (df^i \otimes e_i+f^i\Na(e_i))$. It is easy to check that the definition does not depend on the choice of the frame.

Conversely, a $(1,0)$-connection $\na$ such that $(\na s)(u)$ is holomorphic wherever $s$ and $u$ are holomorphic clearly determines a holomorphic connection.
\end{remark}

Following Atiyah \cite{At}, we set
\[
D(\E)=\E\oplus (\Omega^1\otimes\E),
\]
which is a direct sum as a sheaf of $\C$-vector spaces, endowed with the ${\cal O}$-module structure given by
\[
f\cdot(s,\sigma)=(fs, df\otimes s+f\sigma).
\]
In particular, we have the following exact sequence of (locally free) ${\cal O}$-modules
\begin{equationth}\label{atiyahseq}
O\lra \Omega^1\otimes\E\buildrel{\iota}\over\lra D(\E)\buildrel{\rho}\over\lra \E\lra O.
\end{equationth}%
A \emph{splitting} of this sequence is a morphism $\eta\colon\E\to D(\E)$ of $\O$-modules
such that $\rho\circ\eta=\textrm{id}$. Then

\begin{lemma}[\cite{At}]
\label{split-co}
Let $E$ be a holomorphic vector bundle on a complex manifold~$M$.
A morphism $\eta\colon\E\to D(\E)$ is a splitting of $(\ref{atiyahseq})$ if and only if it is of the form $\eta(s)=\bigl(s,\Na(s)\bigr)$, where $\Na$ is a holomorphic  connection for $ E$. Thus $E$ admits a holomorphic connection if and only if $(\ref{atiyahseq})$ splits.
\end{lemma}

The following is also easy to see\,:

\begin{lemma}[\cite{At}]
Let $\Na$ be a holomorphic connection for a holomorphic vector bundle~$E$. If $\xi\in\Hom_\O(\E,\Omega^1\otimes\E)$
then $\Na+\xi$ is a holomorphic connection for $ E$. Conversely,  every  holomorphic connection for $ E$ is of this form.
\end{lemma}

The sequence (\ref{atiyahseq}) determines an element $b(E)\in
H^1\bigl(M, \Hom(\E,\Omega^1\otimes\E)\bigr)$ as follows. First, applying the functor
$\Hom(\E,\cdot)$ to (\ref{atiyahseq}) we get the exact sequence
\[
O\lra \Hom(\E,\Omega^1\otimes\E)\lra \Hom\bigl(\E,D(\E)\bigr)\lra \Hom(\E,\E)\lra O,
\]
and thus the connecting \homo
\[
\delta\colon H^0\bigl(M, \Hom(\E,\E)\bigr)\lra H^1\bigl(M, \Hom(\E,\Omega^1\otimes\E)\bigr).
\]
Then $b(E)=\delta(\textrm{id})\in H^1\bigl(M, \Hom(\E,\Omega^1\otimes\E)\bigr)$.
It is not difficult to prove the following

\begin{lemma}[\cite{At}]
\label{holoAtiyah} A holomorphic vector bundle $E$ admits a holomorphic connection if and only if $b(E)=0$.
\end{lemma}

Now, we have the Dolbeault \iso
\[
H^1\bigl(M, \Hom(\E,\Omega^1\otimes\E)\bigr)=H^1\bigl(M, \Omega^1\otimes\Hom(\E,\E)\bigr)\simeq
H^{1,1}_{\bar\partial}\bigl(M, \Hom(E,E)\bigr).
\]
Let $a(E)$ denote the class
in $H^{1,1}_{\bar\partial}\bigl(M, \Hom(E,E)\bigr)$ corresponding to $-b(E)$ via the above \iso\ (cf. \cite[Theorem 5]{At}). Then the original Atiyah class of type $(p,p)$ is defined as
\[
a^p_{\textrm{or}}(E)=\left(\frac{\sqrt{-1}}{2\pi}\right)^p\sigma_p(a(E));
\]
we shall show that $a^p_{\textrm{or}}(E)=a^p(E)$ for all $p\ge 1$. To do so we need
some definitions.

\begin{definition} Let $\V=\{V_\lambda\}$ be an open covering of $M$. A \emph{$\V$-splitting} of (\ref{atiyahseq}) is a collection $\{\eta^\lambda\}$ of splittings $\eta^\lambda$ of \eqref{atiyahseq} on each $V_\lambda$. A \emph{holomorphic $\V$-connection} for $ E$ is a collection $\{\Na^\lambda\}$ of holomorphic connections $\Na^\lambda$ for $ E|_{V_\lambda}$.
\end{definition}

By Lemma \ref{split-co}, a  $\V$-splitting determines a holomorphic $\V$-connection and vice versa.
Furthermore, every holomorphic vector bundle $E$ admits a  holomorphic $\V$-connection for some open covering~$\V$. In fact, let $\V=\{V_\lambda\}$ be a covering trivializing $E$; then take as $\Na^\lambda$ a holomorphic connection which is trivial with respect to some holomorphic frame of~$E$ on $V_\lambda$.

\begin{definition}
We shall call \emph{$\bar\partial$-curvature}  of $E$ a $\bar\partial$-closed $(1,1)$-form with coefficients in $\Hom(E,E)$ representing the class $a(E)$.
\end{definition}


The next theorem shows that we can obtain a $\bar\partial$-curvature as the $(1,1)$-component
of the curvature of a suitable $(1,0)$-connection\,:

\begin{theorem}\label{coincidence}
Let $E$ be a holomorphic vector bundle over a complex manifold~$M$. Then
a holomorphic $\V$-connection for $ E$ determines a $(1,0)$-connection $\na$ for $ E$ such that  the $(1,1)$-component of the curvature of $\na$ is a $\bar\partial$-curvature.
\end{theorem}

\begin{proof} Let $\{\Na^\lambda\}$ be a $\V$-connection for $ E$ with respect to
a  (sufficiently fine) open covering $\V=\{V_\lambda\}$ of~$M$.
On $V_\lambda\cap V_\mu$ the difference $\xi^{\lambda\mu}=\Na^\lambda-\Na^\mu$ is an $\O$-morphism from $\E$ to $\Omega^1\otimes\E$, and the collection $\xi=\{\xi^{\lambda\mu}\}$ is a $1$-cocycle  on $\V$ representing  $-b(E)$.

We  denote by $\A^{p,q}\bigl(\Hom(E,E)\bigr)$ the sheaf of germs of smooth forms of type $(p,q)$ with coefficients in the bundle $\Hom(E,E)$; in particular, we may think of $\Hom(\E,\Omega^1\otimes \E)=\Omega^1\otimes\Hom(\E,\E)$ as a subsheaf of $\A^{1,0}\bigl(\Hom(E,E)\bigr)$. Since the sheaf $\A^{p,q}\bigl(\Hom(E,E)\bigr)$ is fine, there exists a $0$-cochain $\{\tau^\lambda\}$ of $\A^{1,0}\bigl(\Hom(E,E)\bigr)$ on $\V$ such that
\[
\xi^{\lambda\mu}=\tau^\mu-\tau^\lambda\quad\text{on}\ V_\lambda\cap V_\mu.
\]
Hence
\[
\Na^\lambda+\tau^\lambda=\Na^\mu+\tau^\mu\quad\text{on}\quad V_\lambda\cap V_\mu.
\]
In this way we have defined a global $(1,0)$-connection $\na$ which coincides with $\Na^\lambda+\tau^\lambda$ on $V_\lambda$.

Since the forms $\xi^{\lambda\mu}$ are holomorphic,
on $V_\lambda\cap V_\mu$ we have $\bar\partial\tau^\lambda=\bar\partial\tau^\mu$. Hence we get a global $\bar\partial$-closed $(1,1)$-form $\omega$ with coefficients in $\Hom(E,E)$ which is equal to $\bar\partial\tau^\lambda$ on~$V_\lambda$. By chasing the diagrams, it is easy to see that the form $\omega$ represents the class $a(E)$, and thus it is a $\bar\partial$-curvature. Moreover, \eqref{curvform} shows
that $\omega$ is the $(1,1)$-component of the curvature of $\na$, and we are done.
\end{proof}

\begin{corollary}\label{sameAtiyah}
Let $E$ be a holomorphic vector bundle on a complex manifold~$M$. Then
\[
a^p_{{\rm or}}(E)=a^p(E)
\]
for all $p\ge1$.
\end{corollary}



\begin{remark}\label{explicit}
Given a holomorphic $\V$-connection $\{\Na^\lambda\}$, the $\bar\partial$-curvature $\omega$
constructed in the proof of Theorem~\ref{coincidence} is not uniquely determined; it depends on the choice of the $0$-cochain
$\{\tau^\lambda\}$. One way to choose $\{\tau^\lambda\}$ is to take a partition of unity
$\{\rho_\lambda\}$ subordinate to $\V$ and set  $\tau^\lambda=\sum_\nu\rho_\nu\xi^{\nu\lambda}$.
\end{remark}

We give now a more explicit expression of the forms introduced in the proof of
Theorem~\ref{coincidence}. Let $\ell$ be the rank of $E$, and choose an open cover $\V$ of
sufficiently small open sets trivializing~$E$. On each $V_\lambda$ take a holomorphic frame
$(e_1^\lambda,\ldots, e_\ell^\lambda)$ of $E$ and let $\Na^\lambda$ be the connection on $V_\lambda$ trivial with respect to this frame. Finally, let $\{h^{\lambda\mu}\}$ be the system of
transition matrices corresponding to these choices, that is
\[
e^\mu_j=\sum_{k=1}^\ell (h^{\lambda\mu})^k_j e^\lambda_k\qquad\textrm{on }V_\lambda\cap
V_\mu.
\]
Then
\[
\xi^{\lambda\mu}(e_i^\lambda)=-\sum_{j,k=1}^\ell (h^{\lambda\mu})^j_{k}\cdot d(h^{\mu\lambda})
_i^k\otimes e_j^\lambda.
\]
Thus $\xi^{\lambda\mu}$ is represented, with respect to  the frame $(e_1^\lambda,\ldots, e_\ell^\lambda)$,  by the matrix
\[
-h^{\lambda\mu}\cdot dh^{\mu\lambda}=dh^{\lambda\mu}\cdot (h^{\lambda\mu})^{-1}=\partial h^{\lambda\mu}\cdot (h^{\lambda\mu})^{-1}
\]
as an element of $\Hom(\E,\Omega^1\otimes\E)\simeq \Omega^1\otimes\Hom(\E,\E)$ on $V_\lambda\cap V_\mu$.
Taking a partition of unity $\{\rho_\lambda\}$ subordinate to $\V$, we may set $\tau^\lambda=\sum_\nu\rho_\nu\xi^{\nu\lambda}$, as in Remark \ref{explicit}; the global $(1,0)$-connection
constructed in the proof of Theorem~\ref{coincidence} is then given by
$\na=\sum_\nu\rho_\nu\Na^\nu$, and its curvature matrix with respect to the frame $(e_1^\lambda,\dots,e_\ell^\lambda)$ is given by $\tau^\lambda$, and the
corresponding $\bar\partial$-curvature by $\bar\partial\tau^\lambda$. 

As a direct consequence of Lemma
\ref{holoAtiyah} and Corollary \ref{sameAtiyah} we get

\begin{proposition}\label{holovanishAtiyah}
Let $E$ be a holomorphic
vector bundle over a complex manifold $M$. If $E$ admits a holomorphic connection
then $a^p(E)=0$ for all $p\geq 1$, that is, all Atiyah classes of $E$ vanish.
\end{proposition}

\begin{remark}
\label{strongvan}
In fact, the existence of a holomorphic connection $\na$ implies the stronger vanishing $a^p(\na)=0$ for all $p\ge 1$.
This can be easily seen from (\ref{curvform}), since the connection matrix $\theta$ of $\na$ with respect to a holomorphic frame is holomorphic. See Theorem \ref{ABvanishing} below for more general vanishing results of this type.
\end{remark}

It should be remarked that the converse of Proposition \ref{holovanishAtiyah} is not true.
Namely, it might happen that $a^p(E)=0$ for all $p\geq 1$ but
$a(E)\neq 0$, as the following example shows.

\begin{example}
Let $M$ be a compact Riemann surface and  $L$  a line
bundle over $M$ such that $a^1(L)=c^1(L)\neq 0$. Let
$E :=L\oplus L^\ast$. Then $c^1(E)=c^1(L)-c^1(L)=0$, and by
Proposition \ref{compactKAtiyah} it follows $a^1(E)=0$. For
dimensional reasons, $a^p(E)=0$ for all $p\geq 2$. Now we claim
that $E$ does not admit a holomorphic connection, and hence
$a(E)\neq 0$ as a class in $H^{1,1}_{\bar\partial}\bigl(M, \Hom(E,E)\bigr)$. In fact, by
contradiction, let $\nabla$ denote a
holomorphic connection for $ E$. Let $\pi\colon \Omega^1\otimes E \to
\Omega^1\otimes L$ denote the projection and $\iota\colon L \to E$
the immersion. It is easy to show that $\pi \circ \nabla \circ
\iota$ is a holomorphic connection for $ L$. But then $c^1(L)=a^1(L)=0$,
against our assumption.
\end{example}

\lsection{\v Cech-Dolbeault cohomology}\label{CechDol}

In this section, we recall the theory of \v Cech-Dolbeault cohomology in the case of coverings consisting of two open sets. Although it is technically more involved, the ideas are similar for the general case of coverings with arbitrary number of open sets. We review relevant material for this case in Section \ref{example} and refer to \cite{Su8} for details

Let $M$ be a complex \mfd\ of dimension $n$. For an open set $U$ of $M$, we denote by $A^{p,q}(U)$ the
vector space of $C^\infty$ $(p,q)$-forms on $U$. Let $\U=\{U_0,U_1\}$ be
an open covering of $M$. We set $U_{01}=U_0\cap U_1$ and define the vector space $A^{p,q}(\U)$ as

$$
A^{p,q}(\U)=A^{p,q}(U_0)\oplus A^{p,q}(U_1)\oplus A^{p,q-1}(U_{01}).
$$
Thus an element $\sigma$ in $A^{p,q}(\U)$ is given by a triple
$\sigma=(\sigma_0, \sigma_1, \sigma_{01})$ with $\sigma_i$
a $(p,q)$-form on $U_i$, $i=0,1$, and
$\sigma_{01}$  a $(p,q-1)$-form on $U_{01}$.

We define a differential operator $\bar D\colon A^{p,q}(\U) \to A^{p,q+1}(\U)$ by
$$
\bar D\sigma=(\bar\partial\sigma_0, \bar\partial\sigma_1, \sigma_1-\sigma_0-\bar\partial\sigma_{01}).
$$
Then we have $\bar D \circ \bar D =0$ and thus a complex for each fixed $p$\,:
\[
\cdots\lra A^{p,q-1}(\U) \buildrel{\bar D^{p,q-1}}\over{\longrightarrow} A^{p,q}(\U)
\buildrel{\bar D^{p,q}}\over{\longrightarrow} A^{p,q+1}(\U) \lra\cdots.
\]

We set
$$
H_{\bar D}^{p,q}(\U)={{\rm Ker}}\, \bar D^{p,q}/{{\rm Im}}\, \bar D^{p,q-1}
$$
and call it the {\it  \v{C}ech-Dolbeault cohomology} of $\U$  of type $(p,q)$.
We denote the image of $\sigma$ by the canonical surjection
${{\rm Ker}}\, \bar D^{p,q}\to H_{\bar D}^{p,q}(\U)$ by $[\sigma]$.

Let $H_{\bar\partial}^{p,q}(M)$ denote the Dolbeault cohomology  of $M$  of type $(p,q)$.

\begin{theorem}\label{DCD}
The map $\a\colon A^{p,q}(M) \to A^{p,q}(\U)$ given by
$\omega \mapsto (\omega, \omega, 0)$
induces an isomorphism
\[
\a\colon H_{\bar\partial}^{p,q}(M) \buildrel{\sim}\over\lra H_{\bar D}^{p,q}(\U).
\]
\end{theorem}

\begin{proof} It is not difficult to show that $\a$ is well-defined.
To prove that $\a$ is surjective, let
$\sigma=(\sigma_0, \sigma_1, \sigma_{01})$ be such that
$\bar D\sigma =0$. Let $\{\rho_0, \rho_1\}$ be a partition of
unity subordinate to the covering $\U$. Define
$\omega=\rho_0\sigma_0+\rho_1\sigma_1-\bar\partial\rho_0\wedge
\sigma_{01}$. Then it is easy to see that $\bar\partial\omega =0$ and
$[(\omega,
\omega, 0)]=[\sigma]$. The injectivity of $\a$ is also not difficult to show.
\end{proof}

We define the \emph{cup product}
\begin{equationth}\label{cup}
A^{p,q}(\U)\times A^{p',q'}(\U)\longrightarrow A^{p+p',q+q'}(\U),
\end{equationth}%
assigning to $\sigma$ in $A^{p,q}(\U)$ and $\tau$ in $A^{p',q'}(\U)$ the element
$\sigma\smallsmile\tau$ in $A^{p+p',q+q'}(\U)$ given by
$$
(\sigma\smallsmile\tau)_0=\sigma_0\wedge\tau_0,\quad(\sigma\smallsmile\tau)_1=\sigma_1\wedge\tau_1\quad\text{and}\quad
(\sigma\smallsmile\tau)_{01}=(-1)^{p+q}\sigma_0\wedge\tau_{01}+\sigma_{01}\wedge\tau_1.
$$
Then $\sigma\smallsmile\tau$ is linear in $\sigma$ and $\tau$ and we have
$$
\bar D(\sigma\smallsmile\tau)=\bar D\sigma\smallsmile\tau+(-1)^{p+q}\sigma\smallsmile\bar D\tau.
$$
Thus it induces the cup product
$$
H^{p,q}_{\bar D}(\U)\times H^{p',q'}_{\bar D}(\U)\longrightarrow
H^{p+p',q+q'}_{\bar D}(\U)
$$
compatible, via the \iso\ of Theorem \ref{DCD}, with the  product in the
Dolbeault cohomology induced from the exterior product of forms.

Now we recall the integration on the \v Cech-Dolbeault cohomology.
Let $M$ and $\U=\{U_0,U_1\}$ be as above and $\{R_0,R_1\}$ a
system of honey-comb cells adapted to $\U$ (cf. \cite{Leh2}, \cite{Su2}).  Thus each $R_i$, $i=0, 1$, is a real sub\mfd\  of dimension $2n$ with $C^\infty$ boundary in $M$ such that $R_i\subset U_i$,  $M=R_0\cup R_1$ and that ${\rm Int}\, R_0\cap {\rm Int} \, R_1=\emptyset$.
 We set $R_{01}=R_0\cap R_1$, which is equal to $\partial R_0=-\partial R_1$ as an oriented \mfd.

Suppose  $M$ is compact; then each $R_i$  is
compact and we may define the integration
\[
\int_M\colon A^{n,n}(\U)\longrightarrow \C
\]
as the sum
$$
\int_M\sigma=\int_{R_0}\sigma_0+\int_{R_1}\sigma_1+\int_{R_{01}}\sigma_{01}
$$
for $\sigma$ in $A^{n,n}(\U)$.
Then this induces the integration on the cohomology
$$
\int_M\colon H^{n,n}_{\bar D}(\U)\longrightarrow \C,
$$
which is compatible, via the \iso\ of Theorem \ref{DCD}, with the usual
integration on the Dolbeault cohomology $H^{n,n}_{\bar\partial}(M)$. Also the bilinear pairing
\[
 A^{p,q}(\U)\times A^{n-p,n-q}(\U)\longrightarrow A^{n,n}(\U)\longrightarrow \C
\]
defined as the composition of the cup product and the integration induces the Kodaira-Serre duality
\begin{equationth}\label{3.4}
KS\colon H^{p,q}_{\bar\partial}(M)\simeq H^{p,q}_{\bar D}(\U)\overset{\sim}{\lra}
H^{n-p,n-q}_{\bar D}(\U)^*\simeq H^{n-p,n-q}_{\bar\partial}(M)^*.
\end{equationth}%

Now let $S$ be a closed set in $M$. Let
$U_0=M\setminus  S$ and $U_1$ a neighborhood of $S$ in $M$, and
consider the covering $\U=\{U_0,U_1\}$ of $M$.
We denote by $A^{p,q}(\U,U_0)$ the subspace of $A^{p,q}(\U)$
consisting of elements $\sigma$ with $\sigma_0=0$,
so that we have the exact sequence
\[
O\lra A^{p,q}(\U,U_0)\lra A^{p,q}(\U)\lra A^{p,q}(U_0)\lra O.
\]

We see that $\bar D$ maps $A^{p,q}(\U,U_0)$ into $A^{p,q+1}(\U,U_0)$.
Denoting by $H^{p,q}_{\bar D}(\U,U_0)$ the $q$-th cohomology of the complex
$(A^{p,\bullet}(\U,U_0), \bar D)$, we have the long exact sequence
\begin{equationth}\label{longexact}
\cdots\lra H^{p,q-1}_{\bar\partial}(U_0)\lra H^{p,q}_{\bar D}(\U,U_0)\lra H^{p,q}_{\bar D}(\U)
\lra H^{p,q}_{\bar\partial}(U_0)\lra \cdots.
\end{equationth}

In view of the fact that $H^{p,q}_{\bar D}(\U)\simeq H^{p,q}_{\bar\partial}(M)$, we set
$$
H^{p,q}_{\bar\partial}(M,M\setminus S)=H^{p,q}_{\bar D}(\U,U_0).
$$

Suppose $S$ is
compact ($M$ may not be) and let $\{R_0,R_1\}$ be a system of honey-comb cells adapted to $\U$.
Then we may assume that $R_1$ is compact and we have the integration
on $A^{n,n}(\U,U_0)$ given by
$$
\int_M\sigma=\int_{R_1}\sigma_1+\int_{R_{01}}\sigma_{01}.
$$
This again induces the integration on the cohomology
$$
\int_M\colon H^{n,n}_{\bar D}(\U,U_0)\longrightarrow \C.
$$

The cup product (\ref{cup}) induces a pairing
$A^{p,q}(\U,U_0)\times A^{n-p,n-q}(U_1)\to A^{n,n}(\U,U_0)$, which,
followed by the integration, gives a bilinear pairing
$$
A^{p,q}(\U,U_0)\times A^{n-p,n-q}(U_1)\longrightarrow \C.
$$
This induces a \homo
\begin{equationth}\label{3.5}
\bar A\colon H^{p,q}_{\bar\partial}(M,M\setminus S)= H^{p,q}_{\bar D}(\U,U_0)\longrightarrow
H^{n-p,n-q}_{\bar\partial}(U_1)^*,
\end{equationth}
that we call the \emph{$\bar\partial$-Alexander \homo.} Note that, although $H^{p,q}_{\bar\partial}(M,M\setminus S)$ does not depend on the choice of $U_1$ because of the exact sequence (\ref{longexact}), $H^{n-p,n-q}_{\bar\partial}(U_1)^*$ does depend on the choice of $U_1$. From the above construction, we have the following

\begin{proposition}\label{commutativity}
If $M$ is compact, the following diagram is commutative\,:
\[
\begin{CD}
H^{p,q}_{\bar\partial}(M,M\setminus S)@>{j^*}>>H^{p,q}_{\bar\partial}(M)\\
  @V{\bar A}VV           @V{\wr}V{KS}V\\
H^{n-p,n-q}_{\bar\partial}(U_1)^*@>{i_*}>>H^{n-p,n-q}_{\bar\partial}(M)^*.
\end{CD}
\]
\end{proposition}


\lsection{Localization of Atiyah classes}

In this section we describe a general scheme for dealing with localization problems.

\subsection{Atiyah classes in the \v Cech-Dolbeault cohomology}\label{AtiyahinCD}

Let $M$ be a complex manifold and $\U=\{U_0,U_1\}$ an open covering of $M$ consisting of two open sets, so that
\[
A^{p,p}(\U)=A^{p,p}(U_0)\oplus A^{p,p}(U_1)\oplus A^{p,p-1}(U_{01}).
\]
For $i=0,1$, let $\na_i$ be a $(1,0)$-connection for $E$ on $U_i$. Then the cochain
\[
a^p(\na_*)=\bigl(a^p(\na_0),a^p(\na_1),a^p(\na_0,\na_1)\bigr)
\]
is in fact a cocycle, because of (\ref{difference}), and thus defines a class $[a^p(\na_*)]$ in $H^{p,p}_{\bar D}(\U)$.

As in the case of Chern classes, it is not difficult to show that the class $[a^p(\na_*)]$ does not depend on the choice of the connections $\na_i$ and corresponds to the Atiyah class $a^p(E)$ via the \iso\ of Theorem \ref{DCD} (cf. \cite[Ch.II, 8. D]{Su2}).

Similarly, if $\varphi$ is a symmetric homogeneous polynomial of degree
$d$, the cocycle
\begin{equationth}\label{CDcocycle}
\varphi^A(\na_*)=\bigl(\varphi^A(\na_0),\varphi^A(\na_1),\varphi^A(\na_0,\na_1)\bigr)
\end{equationth}
defines a class in  $H^{d,d}_{\bar D}(\U)$, which  corresponds to the  class $\varphi^A(E)$ via the \iso\ of Theorem \ref{DCD}.

\subsection{Localization principle}

Let $M$ be a complex manifold of dimension $n$
and $E$ a holomorphic vector bundle of  rank $\ell$ over
$M$.  Also, let $S$ be a closed set in $M$ and $U_1$ a
neighborhood of $S$. Setting $U_0=M\setminus S$, we consider the
covering ${\cal U}=\{U_0,U_1\}$  of $M$. Recall that for a homogeneous
symmetric polynomial $\varphi$ of degree $d$, the characteristic
class $\varphi^A(E)$ in $H^{d,d}_{\bar D}(\U)\simeq H^{d,d}_{\bar\partial}(M)$ is represented by the cocycle $\varphi^A(\na_*)$
in $A^{d,d}({\cal U})$ given by (\ref{CDcocycle}).

It often happens (see, {\it e.g.}, Remark~\ref{strongvan},  Theorems~\ref{vanishing}
and~\ref{ABvanishing} below, or \cite{ABT0, ABT}) that
the existence of a geometric object $\gamma$ on $U_0$ implies the vanishing
of $\varphi(E|_{U_0})$ or of $\varphi^A(E|_{U_0})$, or even of the forms representing them,  for
some symmetric homogeneous polynomial~$\varphi$.
In this section we shall show that in this case we can localize the class $\varphi^A(E)$ at~$S$.

To formalize this idea, assume that given a symmetric homogeneous polynomial $\varphi$ we can associate to $\gamma$ a class ${\cal C}$ of $(1,0)$-connections for $ E|_{U_0}$ such that
$$
\varphi^A(\na)\equiv O
$$
for all $\nabla\in{\cal C}$. We shall also assume (see, {\it e.g.}, Theorem~\ref{ABvanishing})
that
$$
\varphi^A(\na_0,\na_1)\equiv O
$$
for all pairs $\nabla_0$,~$\nabla_1\in\cal C$.
In this case we shall say that $\varphi$ is \emph{adapted} to~$\gamma$,
and we shall call any connection in~$\cal C$ \emph{special.}


Assume that $\na_0$ is special and $\varphi$ is adapted to
$\gamma$. The cocycle $\varphi^A(\na_*)$ is then in
$A^{d,d}({\cal U},U_0)$ and thus it defines a class in
$H^{d,d}_{\bar\partial}(M,M\setminus S)$, which is denoted by
$\varphi^A_S(E,\gamma)$.
It is sent to the  class $\varphi^A(E)$ by the canonical
homomorphism $j^*\colon H^{d,d}_{\bar\partial}(M,M\setminus S)\to H^{d,d}_{\bar\partial}(M)$.
It is not difficult to see that the class $\varphi^A_S(E,\gamma)$
does not depend on the choice of the special connection $\na_0$ or of
the connection~$\na_1$ (cf. \cite[Ch.III, Lemma 3.1]{Su2}). We call $\varphi_S^A(E,\gamma)$ the \emph{localization of
$\varphi^A(E)$ at $S$ by $\gamma$}.

Suppose now $S$ is  compact. Then we have the $\bar\partial$-Alexander
\homo\ (\ref{3.5})
$$
\bar A\colon H^{d,d}_{\bar\partial}(M,M\setminus S)\longrightarrow
H_{\bar\partial}^{n-d,n-d}(U_1)^*.
$$
Thus the class $\varphi^A_S(E,\gamma)$ defines a class in
$H_{\bar\partial}^{n-d,n-d}(U_1)^*$, which we call the \emph{residue} of $\gamma$ for
the class $\varphi^A(E)$ on $U_1$, and denote by ${\rm
Res}_{\varphi^A}(\gamma,E;U_1)$.


Suppose moreover that $S$ has a finite number of connected
components $\{S_\lambda\}_\lambda$. For each $\lambda$, we choose a \nbd\ $U_\lambda$ of $S_\lambda$ so that $U_\lambda\cap U_\mu=\emptyset$ if $\lambda\ne\mu$. Then we have
the residue ${\rm Res}_{\varphi^A}(\gamma,E;U_\lambda)$ in
$H_{\bar\partial}^{n-d,n-d}(U_\lambda)^*$ for each $\lambda$.
Let $R_\lambda$ be a $2n$-dimensional manifold with
$C^{\infty}$ boundary in $U_\lambda$ containing  $S_\lambda$ in
its interior and set $R_{0\lambda}=-\partial R_\lambda$.
Then the residue ${\rm Res}_{\varphi^A}(\gamma,E;U_\lambda)$ is
represented by a functional
\begin{equationth}\label{reshom}
\eta\mapsto\int_{R_\lambda}\varphi^A(\na_1)\wedge\eta
+\int_{R_{0\lambda}}\varphi^A(\na_0,\na_1)\wedge\eta
\end{equationth}%
for every $\bar\partial$-closed $(n-d,n-d)$-form $\eta$ on $U_\lambda$.


From the above considerations and Proposition \ref{commutativity}, we have
the following \emph{residue theorem}\,:

\begin{theorem}\label{residueth}
Let $E$ be a holomorphic vector bundle on a complex manifold $M$ of dimension~$n$. Let $S$ be a compact
subset of~$M$ with a finite number of connected components~$\{S_\lambda\}_\lambda$.
Assume we have a geometric object $\gamma$ on $U_0=M\setminus S$ and a symmetric
homogeneous polynomial~$\varphi$ of degree~$d$, adapted to $\gamma$. For each~$\lambda$ choose a
neighbourhood~$U_\lambda$ of~$S_\lambda$ so that $U_\lambda\cap U_\mu=\emptyset$
when $\lambda\ne\mu$. Then\,:
\begin{enumerate}
\item[{\rm(1)}] For each connected component $S_\lambda$ the residue ${\rm Res}_{\varphi^A}(\gamma,E;U_\lambda)$
in  the dual space $H_{\bar\partial}^{n-d,n-d}(U_\lambda)^*$ is
represented by the functional $(\ref{reshom})$;
\item[{\rm(2)}] if moreover $M$ is compact, then
$$
\sum_\lambda (i_\lambda)_*{\rm
Res}_{\varphi^A}(\gamma,E;U_\lambda)=KS(\varphi^A(E))
\qquad\text{in}\ \ H_{\bar\partial}^{n-d,n-d}(M)^*,
$$
where $i_\lambda\colon U_\lambda \hookrightarrow M$ denotes the
inclusion.
\end{enumerate}
\end{theorem}

\begin{remark}\label{resnum}
If $d=n$ and $M$ is compact and connected, then $H_{\bar\partial}^{n-d,n-d}(M)^*=H_{\bar\partial}^{0,0}(M)^*$ may be identified with $\C$, and in this case, $ (i_\lambda)_*{\rm
Res}_{\varphi^A}(\gamma,E;U_\lambda)$ is a complex number given by
\[
\int_{R_\lambda}\varphi^A(\na_1)
+\int_{R_{0\lambda}}\varphi^A(\na_0,\na_1),
\]
and $KS(\varphi^A(E))$ may be expressed as $\int_M \varphi^A(E)$.

Furthermore, in this case $H_{\bar\partial}^{0,0}(M)^*=H_0(M,\C)$,
and $\varphi^A$ may be replaced by $\varphi$ (cf. Remark \ref{top}) so that the Atiyah residue equals the Chern residue.
\end{remark}

We finish this section by studying what happens in the case of compact K\"ahler manifolds.
Thus let $M$ be a compact K\"ahler manifold of dimension $n$, and
$E$ a holomorphic vector bundle on $M$.  We have the
following commuting diagram\,:
\[
\begin{CD}
H^{p,p}_{\bar\partial}(M)@>{H}>>H^{2p}_{d}(M)\\
  @V{KS}V{\wr}V           @V{\wr}V{P}V\\
H^{n-p,n-p}_{\bar\partial}(M)^*@>{H_*}>>H_{2n-2p}(M,\C).
\end{CD}
\]
where $H$ denotes the injection given by the Hodge decomposition, $H_*$  the injection given by the dual decomposition, and
$P$ the Poincar\'e isomorphism, which is given by the cap product with the fundamental cycle $[M]$.

Since
$H(\varphi^A(E))=\varphi(E)$ in this case (Proposition \ref{compactKAtiyah}), applying $H_*$ to the both sides of the formula in Theorem \ref{residueth}.(2), we actually have a localization
result for Chern classes\,:

\begin{theorem}\label{loc-K}
Let $E$ be a holomorphic vector bundle on a compact K\"ahler manifold $M$ of dimension~$n$. Let $S$ be a compact
subset of~$M$ with a finite number of connected components~$\{S_\lambda\}_\lambda$.
Assume we have a geometric object $\gamma$ on $U_0=M\setminus S$ and a symmetric
homogeneous polynomial~$\varphi$ of degree~$d$, adapted to $\gamma$. For each~$\lambda$ choose a
neighbourhood~$U_\lambda$ of~$S_\lambda$ so that $U_\lambda\cap U_\mu=\emptyset$
when $\lambda\ne\mu$. Then
$$
\sum_\lambda H_*\bigl( (i_\lambda)_*{\rm Res}_{\varphi^A}(\gamma,E;U_\lambda)\bigr)=\varphi(E)\frown [M]
\qquad\text{in}\ \ H_{2n-2d}(M,\C).
$$
\end{theorem}

Notice that $H_* \bigl((i_\lambda)_*{\rm Res}_{\varphi^A}(\gamma,E;U_\lambda)\bigr)$ is represented by a cycle $C$ such that
for each closed $(2n-2d)$-form $\omega$, the integral $\int_C\o$ is given by the right-hand side of (\ref{reshom}) with $\eta$ a $\bar\partial$-closed
$(n-d,n-d)$-form representing the $(n-d,n-d)$-component of the class $[\o]\in H^{2d-2n}_d(M)$.

\lsection{Localization by  frames}\label{locfr}

In this section we give a first example of localization of Atiyah classes following the
scheme indicated in the previous section.

The starting point is the following vanishing theorem, which is a consequence of
the corresponding vanishing theorem for Chern forms (cf., {\it e.g.}, \cite[Ch.II, Proposition 9.1]{Su2}).

\begin{theorem}\label{vanishing}
Let $E$ be a holomorphic vector bundle of rank~$\ell$ on a complex manifold~$M$.
Let $s^{(r)}=(s_1,\dots,s_r)$ be an $r$-frame of $E$ on an open set $U\subset M$,
and $\na$ an $s^{(r)}$-trivial  $(1,0)$-connections for $E$ on $U$. Then
$$
a^p(\na)=O,\qquad\text{on}\ \ U\ \text{ for }\ p\ge \ell-r+1.
$$
\end{theorem}

Let $S$ be a closed set in $M$ and assume we have an $r$-frame $s^{(r)}$ of $E$ on $M\setminus S$. We let $U_0=M\setminus S$, choose a neighborhood $U_1$  of $S$, and consider the covering $\U=\{U_0, U_1\}$ of $M$.  Let $\na_0$ be an $s^{(r)}$-trivial $(1,0)$-connection for $E$ on $U_0$, and $\na_1$ an arbitrary
$(1,0)$-connection for $E$ on $U_1$. The $p$-th Atiyah class $a^p(E)$ is represented by the \v Cech-Dolbeault cocycle
\[
a^p(\na_*)=\bigl(a^p(\na_0),a^p(\na_1),a^p(\na_0,\na_1)\bigr).
\]
By Theorem \ref{vanishing}, if $p\ge \ell-r+1$, we have $a^p(\na_0)=0$; thus $a^p(\na_*)\in A^{p,p}(\U,U_0)$ determines a class  in $H^{p,p}_{\bar\partial}(M,M\setminus S)$, which we denote by $a^p(E,s^{(r)})$ and  call  the \emph{localization of $a^p(E)$ by $s^{(r)}$.}

\begin{remark} If we have several $s^{(r)}$-trivial $(1,0)$-connections, we also have the vanishing of their difference form, and so $s^{(r)}$-trivial $(1,0)$-connections are special in the sense
discussed in the previous section. As a consequence,  the localization $a^p(E,s^{(r)})$ does not depend on the choice of the $s^{(r)}$-trivial $(1,0)$-connection $\na_0$ (or of the $(1,0)$-connection $\na_1$); cf. \cite{Su2}.
\end{remark}

\begin{example} Let $C$ be a compact Riemann surface and $L$ a holomorphic line bundle over $C$. Suppose we have a meromorphic section $s$ of $L$ and let $S$ be the set of zeros and poles of $s$. The previous construction gives us the localization $a^1(L,s)$ in $H^{1,1}_{\bar\partial}(C,C\setminus S)$ of $a^1(L)$ in $H^{1,1}_{\bar\partial}(C)$. Note that $S$ consists of a finite number of points. Let $p$ be a point in $S$ and choose an open \nbd\ $U$ of $p$ not containing any other point in $S$ and trivializing~$L$. Let $e$ be a holomorphic frame of~$L$
on $U$, and write $s=fe$ with $f$ a meromorphic function on $U$. Let $\na_0$ be the $s$-trivial connection for $L$ on $C\setminus S$ and $\na_1$ the $e$-trivial connection for $L$ on $U$. If we denote by $i$ the embedding $U\hookrightarrow C$, we have (by Theorem~\ref{residueth} and Remark~\ref{resnum})
\[
i_*{\rm Res}_{a^1}(L,s;U)=\int_R a^1(\na_1)-\int_{\partial R}a^1(\na_0,\na_1).
\]
But we also have $a^1(\na_1)=0$, and a computation gives
\[
a^1(\na_0,\na_1)=\frac{\sqrt{-1}}{2\pi}\frac{df}f.
\]
So
\[
i_*{\rm Res}_{a^1}(L,s;U)=\frac{1}{2\pi\sqrt{-1}}\textrm{Res}_p\left(\frac{df}{f}\right),
\]
and Theorem~\ref{loc-K} yields
\[
\sum_{p\in S} \frac1{2\pi\sqrt{-1}}\textrm{Res}_p\left(\frac{df}f\right)=\int_C a^1(L).
\]
In particular we have recovered the classical residue formula for the Chern class, as $\int_C c^1(L)=\int_C a^1(L)$ in this case.
\end{example}

See \cite{Su8} for another fundamental example of localized classes of this type, {\it i.e.}, the ``$\bar\partial$-Thom class" of a holomorphic \vb.

\lsection{A Bott type vanishing theorem}\label{ABvanish}

Let $M$ be a complex  manifold  and $E$ a complex vector bundle over $M$.
If $H$ is a subbundle of the complexified tangent bundle $T^c_\R M$, then its dual  $H^*$ is
canonically viewed as a quotient
of $(T^c_\R M)^*$. We denote by $\rho$ the canonical projection $(T^c_\R M)^*\to H^*$.
Following \cite{BB2}, we give the following definition.

\begin{definition}\label{partialconnection}
A {\it partial connection} for $ E$ is a pair $(H,\delta)$ given by a subbundle $H$
of $T^c_\R M$ and a $\C$-linear map
\[
\delta\colon A^0(M,E)\lra A^0(M,H^*\otimes E)
\]
satisfying
$$
\delta(fs)=\rho(df)\otimes s+f\delta(s)\qquad \text{for}\ f\in A^0(M)\ \text{and}\ s\in A^0(M,E).
$$
\end{definition}

As in the case of connections, it is easy to show that a partial connection is a local operator and
thus it admits locally a representation by a matrix whose entries are $C^{\infty}$ sections of $H^*$.

\begin{definition}\label{conext} Let $(H,\delta)$ be a partial connection for $ E$. We say that a connection
$\na$ for $ E$ \emph{extends} $(H,\delta)$ if the diagram
\[
\begin{CD}
A^0(M,E) @>\na>> A^1(M,E)=A^0(M,(T^c_\R M)^*\otimes E) \\
@VidVV                         @V\rho\otimes1VV \\
A^0(M,E) @>\delta>> A^0(M,H^*\otimes E)
\end{CD}
\]
is commutative.
\end{definition}

It is easy to see that the following lemma holds (\cite[Lemma (2.5)]{BB2}).

\begin{lemma} Any partial connection for a complex vector bundle admits an extension.\end{lemma}

\begin{example} If $E$ is  holomorphic, then we have the differential operator
\[
\bar\partial\colon A^0(M,E)\lra  A^0(M,\overline T^*M\otimes E).
\]
The pair $(\overline TM,\bar\partial)$ is a partial connection for $E$.
\end{example}

The following is not difficult to prove\,:

\begin{lemma}[\cite{BB2}]
A connection $\na$ for a holomorphic \vb\ $E$ is of type $(1,0)$ if and only if it  extends $(\overline TM,\bar\partial)$.
\end{lemma}

\begin{definition} Let $E$ be a holomorphic \vb\ over $M$.
A {\it holomorphic partial connection} for $ E$ is a pair
$(F,\mbox{\boldmath $\delta$})$ given by  holomorphic subbundle $F$ of $TM$ and a
$\C$-linear homomorphism
\[
\mbox{\boldmath $\delta$} : \E\lra \F^\ast\otimes \E
\]
satisfying
\[
\mbox{\boldmath $\delta$}(fs)=\rho(df)\otimes s+f\mbox{\boldmath $\delta$}(s)\quad\text{for}\quad f\in{\cal O}\  \text{and}\  s\in\E.
\]
We shall also say that $\mbox{\boldmath $\delta$}$ is a holomorphic partial connection
\emph{along} $F$.
\end{definition}

\begin{remark}\label{holopartial}
A holomorphic partial connection $(F,\mbox{\boldmath $\delta$})$ for a
holomorphic vector bundle $E$ induces a partial connection in the sense of Definition \ref{partialconnection}  (cf. Remark \ref{holoCinfty}). Conversely, if  $(F,\delta)$ is a ($C^\infty$) partial connection such that
$\delta(s)(u)$ is holomorphic wherever $s$ and $u$ are holomorphic,
then it defines a holomorphic partial connection, and we shall say that $(F,\delta)$ is \emph{holomorphic.}

Note that, if there is an ``action" of $F$ on $E$, it naturally defines a partial connection for $E$ along $F$ (cf. \cite[Ch.II, 9]{Su2}).
\end{remark}

\begin{remark}\label{holopartialdue}
A holomorphic connection $\Na$ on  $E$ clearly gives a holomorphic partial connection $(TM,\Na)$. The connection $\na$ in Remark~\ref{holoCinfty} (that is, $\Na$ viewed as a $C^\infty$ connection) is a connection extending $(TM\oplus\overline{T}M,\Na\oplus\bar\partial)$.
\end{remark}

\begin{definition} Let $(F,\delta)$ be a partial holomorphic connection for $E$. An {\it $F$-connection} for $E$ is a connection for $E$ extending $(F\oplus \bar TM,\delta\oplus\bar\partial)$.
\end{definition}
Using holomorphic partial connections we have a vanishing theorem generalizing
Pro\-po\-si\-tion~\ref{holovanishAtiyah}\,:

\begin{theorem}\label{ABvanishing}
Let $M$ be a complex \mfd\ of dimension $n$ and $F$ a
holomorphic  subbundle of rank $r$ of $TM$. Let $E$ be a holomorphic \vb\ over $M$ and
$(F,\delta)$ a holomorphic partial connection for $E$.
 If $\na_0,\dots,\na_q$ are $F$-connections for $E$, then
 $$
\varphi^A(\na_0,\dots,\na_q)\equiv O
$$
for all homogeneous
symmetric polynomials $\varphi$ of degree $d>n-r$.
\end{theorem}

\begin{proof} For simplicity, we prove the theorem for the case  $q=0$. The case for general
$q$  follows from the construction of the difference form (see \cite{Su8}).

Thus let $\na$  be an $F$-connection for $E$.  Note that the problem is local; so choose a  holomorphic frame $s^{(\ell)}=(s_1,\dots,s_\ell)$ of
$E$ on some open set $U$, and let $\theta$ be the connection matrix of $\na$ with respect to  $s^{(\ell)}$. Taking a smaller $U$, if necessary, we may write $TM=F\oplus G$ for some holomorphic vector bundle $G$ of rank $n-p$ on $U$. We have the corresponding decomposition $T^*M=F^*\oplus G^*$. Taking, again if necessary, a smaller $U$, we can choose a holomorphic frame $u^{(r)}=(u_1,\dots,u_r)$ of $F$ on $U$. Let $(u^*_1,\dots,u^*_r)$ be the holomorphic frame of $F^*$ dual
to $u^{(r)}$ and $(v^*_1,\dots,v^*_{n-r})$ a holomorphic frame of $G^*$ on $U$.
Since $\na$ is of type $(1,0)$, each entry of $\theta$ may be written as
$\sum_{j=1}^pa^j u^*_j+\sum_{k=1}^{n-r}b^k v^*_k$ with $a^j$,~$b^k\in C^\infty(U)$.  By definition, we have $\na(s_i)(u_j)=\delta(s_i)(u_j)$, which is holomorphic. Thus each $a^j$ is holomorphic and hence the corresponding entry of $\kappa^{1,1}=\bar\partial\theta$ is of the form
\[
\sum_{k=1}^{n-r}\bar\partial b^k\wedge v^*_k,
\]
which  yields the theorem.
\end{proof}

Another proof of the same theorem can be given along the lines of the original
Bott vanishing theorem and of \cite[Theorem~6.1]{ABT}\,:
\smallbreak
\begin{proof}[Second proof of Theorem \ref{ABvanishing}]
Let $\na$ and $TM=F\oplus G$ be chosen as in the previous proof.
The curvature $K$ of~$\nabla$ satisfies
\[
K(X,\overline{Z})=0
\]
for all sections $X$ of $F$ and $\overline{Z}$ of $\overline{T}M$. Hence, if
$\{u_1^*, \ldots, u_r^*, v_1^*,\ldots, v_{n-r}^*,
d\bar{z}_1,\ldots, d\bar{z}_n\}$ is a basis of
$(T_\R^c M)^\ast$ with respect to the decomposition
$T_\R^c M\otimes=F\oplus G\oplus \bar TM$, it follows that the
$(1,1)$-part of each entry of the curvature matrix of $K$  in
such a frame is of the form
\[
\sum_{j=1}^{n-r}\sum_{k=1}^n
a^j_k v^*_j\wedge d\bar{z}_k,
\]
and again the assertion follows.
\end{proof}

\begin{remark} \label{pvt}
The previous vanishing theorem is the analogous of the Bott
vanishing theorem for Chern forms. As shown in \cite[Theorem
6.1]{ABT}, under the same hypotheses we have
$\varphi(\na)=0$ for a symmetric homogeneous polynomial
$\varphi$ of degree $d> n-r+[\frac r 2]$, where $[q]$ denotes the integer part of~$q$.

See Section \ref{example} below for an example where the Atiyah form vanishes but the corresponding Chern form does not.
\end{remark}

\begin{remark}
A version of this Bott type vanishing theorem for Atiyah classes is
proved in \cite[Proposition (3.3)]{BB1} and \cite[Proposition 5.1]{CL} by  cohomological
arguments (actually, in the latter the authors assume $F$
to be involutive, but involutiveness is not really needed in
their argument). The above theorem gives a more precise form of the vanishing
theorem in the sense that it gives the vanishing at the form level.
\end{remark}

Putting together Theorem \ref{ABvanishing}, Remark~\ref{pvt} and  Propositions~\ref{compactKAtiyah} and~\ref{holovanishAtiyah}, we have

\begin{theorem}\label{null}
Let $E$ be a holomorphic vector bundle on a complex manifold~$M$.
Assume that $E$ admits a holomorphic connection $\Na$, and let $\na$ be correspoonding
$(1,0)$-connection (cf. Remark \ref{holopartialdue}). Let $\varphi$ be a a symmetric homogeneous polynomial  of degree $d>0$.
Then $\varphi^A(\na)=0$. Moreover if $d>[\frac n 2]$, then $\varphi(\na)=0$.
Furthermore, if $M$ is compact K\"ahler then $\varphi(E)=0$ always.
\end{theorem}

\lsection{Partial connection for the normal bundle of an invariant submanifold}
\label{CSaction}

Let $M$ be a complex manifold.  A (non-singular holomorphic) distribution on $M$ is a holomorphic subbundle $F$  of $TM$.
The rank of the distribution is the rank of $F$.
In this section, we construct a partial connection for the normal bundle of an invariant submanifold of a distribution.

Let $V$ be a complex sub\mfd\  of $M$. We  denote by ${\cal I}_V\subset\O$ the idealsheaf of holomorphic function germs
vanishing on~$V$ so that ${\cal O}_V={\cal O}/{\cal I}_V$ is the sheaf of germs of holomorphic functions on $V$.
Denoting by $N_V$  the normal bundle of $V$ in $M$, we have the exact sequence
\[
O\lra TV\lra TM|_V\buildrel\pi\over\lra N_V\lra O.
\]

We say that a distribution $F$ on $M$ leaves $V$ invariant (or $F$ is tangent to $V$), if $F|_V\subset TV$.

\begin{theorem}\label{holo-normal}
Let $V$ be a complex submanifold of $M$. If a distribution $F$ on $M$ leaves $V$ invariant, there exists a holomorphic partial connection $\mbox{\boldmath $\delta$}$ for the
normal bundle $N_V$ along $F|_V$.
\end{theorem}

\begin{proof}
Let $x$ be a point in $V$ and take $u\in {\cal O}_V(F|_V)_x$ and  $s\in {\cal O}_V(N_V)_x$. Let $\tilde{u}\in {\cal F}_x$ and
$\widetilde{s}\in\Theta_x$ such that $\tilde{u}|_V=u$ and
$\pi(\widetilde{s}|_V)=s$, where $\pi\colon {\cal O}_V(TM|_V)\to {\cal O}_V(N_V)$ is the
natural projection. Define $\mbox{\boldmath $\delta$} :{\cal O}_V(N_V)\lra {\cal O}_V(F|_V)^*\otimes {\cal O}_V(N_V)$ by
\[
\mbox{\boldmath $\delta$}(s)(u) :=\pi([\tilde{u}, \widetilde{s}]|_V).
\]
It is easy to show that $\mbox{\boldmath $\delta$}$ does not depend on the choice
of $\widetilde{s}$. As for $\tilde{u}$, let $F$ be locally
generated by
holomorphic sections $\tilde v_1, \ldots, \tilde v_r$ of $TM$, where $r=\textrm{rank}\,F$. Choose
local coordinates $\{z_1,\ldots,
z_n\}$ on $M$ such that $V=\{z_{m+1}=\ldots=z_n=0\}$. We shall denote by $T_k$
any local vector field of the form $\sum_{j=1}^m a^j
\frac{\partial}{\partial z_j}$ with $a_j\in \mathcal I_V^k$
(where clearly $\mathcal I_V^0=\mathcal O$); by $N_k$ any local vector field of
the form $\sum_{j=m+1}^n a^j \frac{\partial}{\partial z_j}$ with
$a^j\in \mathcal I_V^k$; and by $R_k$ any local vector field of the form
$\sum_{j=1}^n a^j \frac{\partial}{\partial z_j}$ with $a^j\in
\mathcal I_V^k$.

Since $F|_V\subset TV$, it follows that $\tilde
v_j=T_0+N_1+R_2$ for $j=1,\ldots, r$. Therefore, since the
rank of $F$ and the rank of $F|_V$ are the same,  if
\[
u=\sum_{j=1}^r g^j \tilde v_j|_V
\]
with $g^j\in \mathcal O_V$, then
\[
\tilde{u}=\sum_{j=1}^r \tilde g^j \tilde v_j
\]
with $\tilde g^j\in \mathcal O$ such that
$\tilde g^j|_V=g^j$. Denoting by
$g^j$ the natural extension $(z_1,\ldots, z_n)\mapsto
g^j(z_{m+1},\ldots, z_n)$, it follows that
\[
\tilde g^j-g^j=h^j\in\mathcal I_V.
\]
Hence
\[
\tilde{u}=\sum_{j=1}^r g^j \tilde v_j + \sum_{j=1}^r
h^j \tilde v_j
\]
But
\[
h^j \tilde v_j=h^j (T_0+N_1+R_2)=T_1+R_2,
\]
and it is easy to see that this latter term does not give any contribution to the
expression $\pi([\tilde{u}, \widetilde{s}]|_V)$. From this it
follows that $\mbox{\boldmath $\delta$}$ is well defined, and it is easy to check that it is a
holomorphic partial connection.
\end{proof}

Note that the above partial connection  is already known for foliations (cf. {\it e.g.}, \cite{LS}).
From Theorems~\ref{holo-normal} and~\ref{ABvanishing}, we have

\begin{corollary}\label{cor0} Let $V$ be a complex submanifold of $M$ of dimension $m$ and $F$ a distribution on $M$ of rank $r$ leaving $V$ invariant. Also let $\na$ be a $(1,0)$-connection  for $N_V$ extending the partial connection $\mbox{\boldmath $\delta$}$ of Theorem \ref{holo-normal}.
Then $\varphi^A(\na)=O$ for all symmetric homogeneous
polynomial $\varphi$ of degree $d> m-r$.
\end{corollary}

We also get the following obstruction to the existence of
distributions (not necessarily integrable) tangent to a given
submanifold\,:

\begin{corollary}\label{cor}
Let $V$ and $F$ be as in Corollary \ref{cor0}.
Then $\varphi^A(N_V)=O$ for all symmetric homogeneous
polynomial $\varphi$ of degree $d> m-r$.

Moreover, if $V$ is compact K\"ahler then we have
$\varphi(N_V)=O$ for all symmetric homogeneous polynomial
$\varphi$ of degree $d> m-r$.
\end{corollary}

\lsection{Residues of singular distributions}\label{singdist}

A general theory of singular holomorphic distributions can be developed modifying the one for singular holomorphic foliations  (cf. \cite{BB2}, \cite[Ch.VI]{Su2}), omitting the integrability condition.

Let $M$ be a complex \mfd\ of dimension $n$.   For simplicity, we assume that $M$ is connected.

\begin{definition} A (singular)  {\it holomorphic distribution} of rank $r$ on $M$  is a coherent sub-${\cal O}_M$-module ${\cal F}$ of  rank $r$ of $\Theta$.
\end{definition}

In the above, the rank of $\F$ is the rank of its locally free part. Note that, since $\Theta$ is locally free, the coherence of ${\cal F}$ here simply means that it is locally finitely generated. We call ${\cal F}$ the {\it tangent sheaf} of the distribution and the quotient ${\cal N}_{\cal F}=\Theta/{\cal F}$ the {\it normal sheaf} of the distribution.

The singular set $S({\cal F})$ of a distribution ${\cal F}$ is defined to be the singular set of the coherent sheaf
${\cal N}_{\cal F}$\,:
\[
S({\cal F})={\rm Sing}({\cal N}_{\cal F})=\{\, x\in M\mid {{\cal N}_{\cal F}}_x\ \text{is not}\ {\cal O}_x\text{-free}\,\}.
\]

Note that ${\rm Sing}({\cal F})\subset S({\cal F})$. Away from $S({\cal F})$, the sheaf ${\cal F}$ defines a non-singular distribution of rank $r$.

In particular, if ${\cal F}$ is locally free of rank $r$, in a \nbd\ of each point in $M$ it is generated by $r$ holomorphic vector fields $v_1,\dots, v_r$, without relations, on $U$. The set $S({\cal F})\cap U$ is the set of points where the vector fields fail to be linearly independent.

Singular distributions can be dually defined in terms of cotangent sheaf. Thus a {\it singular distribution of corank $q$} is a coherent subsheaf $\G$ of rank $q$ of $\Omega^1$. Its annihilator
\[
\F=\G^a=\{\, v\in\Theta\mid \langle v,\o\rangle=0\ \ \text{for all} \ \o\in\G\,\}
\]
 is a singular distribution of rank $r=n-q$.

Corollary \ref{cor} in the previous section has a slightly stronger version when
the rank of the distribution is equal to the dimension of the
submanifold. Namely

\begin{proposition}
Let $V\subset M$ be a complex submanifold of dimension $m$. Let
${\cal F}$ be a (possibly singular) holomorphic distribution of
rank $m$. Assume that ${\cal F}\otimes{\cal O}_V\subset {\cal O}_V(TV)$ and that $\Sigma=S({\cal F})\cap V$ is  an analytic
subset of $V$ of codimension at least $2$. Then $a^p(N_V)=O$
for all $p>0$.

Moreover, if $V$ is compact K\"ahler then $c^p(N_V)=O$ for all
$p>0$.
\end{proposition}

\begin{proof}
We shall show that there exists a holomorphic connection for
$N_V$, then the result follows from Theorem \ref{null}.

By Theorem \ref{holo-normal} there exists a holomorphic
connection $\na$ for $N_V$ on $V\setminus \Sigma$. We are going to
prove that such a connection extends holomorphically through
$\Sigma$. Indeed, let $p\in \Sigma$. Let $U$ be an open
neighborhood of $p$ in $V$ such that $N_V|_U$ is trivial. Let
$e_1,\ldots, e_k$ be a holomorphic frame for $N_V|_U$ (here
$k=\dim M-m$). Let $\omega$ be the connection matrix of $\nabla$ on $U\setminus\Sigma$.
With respect to
local coordinates $(z_1,\ldots, z_m)$ on $U$, the entries of
$\omega$ are $(1,0)$-forms of the type $\sum_j a_j(z) dz_j$ with
$a_j\colon  U\setminus\Sigma\to \C$ holomorphic. Since  $\Sigma$ has
codimension at least two in $U$, Riemann's extension theorem
implies that each $a_j$ admits a (unique) holomorphic extension to $U$.
In this way we have extended $\nabla$ over~$U$, and hence
$N_V$ admits a holomorphic connection.
\end{proof}

Now suppose $\F$ is a singular distribution of rank $r$ and set $U_0=M\setminus S$ and $S=S(\F)$. Let $U_1$ be a neighborhood of $S$ in $M$ and consider the covering ${\cal U}=\{U_0,U_1\}$. On $U_0$, we have a subbundle
$F_0$ of $TM$ such that $\F|_{U_0}={\cal O}(F_0)$.

Suppose $E$ is a holomorphic \vb\ on $M$ admitting a partial holomorphic connection $(F_0,\delta)$ on $U_0$. Then, choosing an $F_0$-connection $\na_0$ on $U_0$ and a $(1,0)$-connection $\na_1$ on $U_1$, for a symmetric homogeneous polynomial $\varphi$ of degree $d>n-r$, we have the localization $\varphi^A(E,\F)$ in $H_{\bar D}^{d,d}({\cal U}, U_0)$ of $\varphi^A(E)$ in $H_{\bar D}^{d,d}({\cal U})\simeq H_{\bar\partial}^{d,d}(M)$ and, via the $\bar\partial$-Alexander homomorphism, the corresponding residues.

We restate the residue theorem (Theorem \ref{residueth})  in this context\,:

\begin{theorem}\label{residuethdist} In the above situation, suppose $S$ has a finite number of connected components~$\{S_\lambda\}_\lambda$.
 Then\,:
\begin{enumerate}
\item[{\rm(1)}] For each $\lambda$ we have the residue ${\rm Res}_{\varphi^A}(\F,E;U_\lambda)$
in   $H_{\bar\partial}^{n-d,n-d}(U_\lambda)^*$;
\item[{\rm(2)}] if  $M$ is compact, then
$$
\sum_\lambda (i_\lambda)_*{\rm
Res}_{\varphi^A}(\F,E;U_\lambda)=KS(\varphi^A(E))
\qquad\text{in}\ \ H_{\bar\partial}^{n-d,n-d}(M)^*.
$$
\end{enumerate}
\end{theorem}

\lsection{An example}\label{example}

In this section, we give an example of the Atiyah residue  of a singular distribution on the normal bundle of an invariant submanifold.

We start with the $1$-form

\[
\o=z\,dx+z\,dy-y\,dz
\]
on $\C^3$ with coordinates $(x,y,z)$. It defines a corank one singular  distribution on $\C^3$ with singular set $\{y=z=0\}$. As generators of its annihilator,  we may take the vector fields
\begin{equationth}\label{vfgen}
v_1=y\frac \partial {\partial y}+z\frac \partial {\partial z}\quad\text{and}\quad v_2=\frac \partial {\partial x}-\frac \partial {\partial y}.
\end{equationth}

 It leaves the plane $\{z=0\}$ invariant. Note that from $\o\wedge d\o=-z\,dx\wedge dy\wedge dz$, we see that $\o$ defines a contact structure on $\C^3$ with singular set $\{z=0\}$ (Martinet hypersurface). We will see that the first Atiyah class of the normal bundle of the (projectivized) Martinet hypersurface is localized at the singular set of the corresponding distribution.

Now we projectivize everything. Thus let $\P^3$ be the complex projective space of dimension three with homogeneous coordinates $\z=(\zeta_0 :\zeta_1 :\zeta_2 :\zeta_3)$. The projective space $\P^3$ is covered by four open sets $W^{(i)}$, $0\le i\le 3$, given by $\zeta_i\ne 0$. We take the original affine space $\C^3$  as $W^{(0)}$ with $x=\zeta_1/\zeta_0$, $y=\zeta_2/\zeta_0$ and $z=\zeta_3/\zeta_0$.

We consider the corank one distribution $\G$ on $\P^3$ naturally obtained as an extension of the above\,:

\vv

\noindent(0)
On $W^{(0)}$, ${\cal G}$ is defined by $\o_0=z\,dx+z\,dy-y\,dz$ as given before.

\vv

\noindent(1) On $W^{(1)}$, we set $x_1=\z_0/\z_1$, $y_1=\z_3/\z_1$ and $z_1=\z_2/\z_1$. Then ${\cal G}$ is defined by
\[
\o_1=-y_1\,dx_1-x_1z_1\,dy_1+x_1y_1\,dz_1.
\]

\vv

\noindent(2) On $W^{(2)}$, we set $x_2=\z_3/\z_2$, $y_2=\z_0/\z_2$ and $z_2=\z_1/\z_2$. Then ${\cal G}$ is defined by
\[
\o_2=-y_2\,dx_2-x_2z_2\,dy_2+x_2y_2\,dz_2.
\]

\vv

\noindent(3) On $W^{(3)}$, we set $x_3=\z_2/\z_3$, $y_3=\z_1/\z_3$ and $z_3=\z_0/\z_3$. Then ${\cal G}$ is defined by
\[
\o_3=z_3\,dx_3+z_3\,dy_3-y_3\,dz_3.
\]

Note that $\o_i=(\z_j/\z_i)^3\o_j$ in $W^{(i)}\cap W^{(j)}$ so that the conormal sheaf of the distribution $\G$ is locally free of rank one and, as a line bundle, it  is three times  the hyperplane bundle on $\P^3$.
Let $\F=\G^a$ be the annihilator of $\G$, which defines a singular distribution of rank two on $\P^3$.
The  singular set $S(\F)$ of $\F$, which coincides with that of $\G$, has three irreducible components $S_1=\{\zeta_2=\zeta_3=0\}$, $S_2=\{\zeta_0=\zeta_3=0\}$ and $S_3=\{\zeta_0=\zeta_1=0\}$. We have a subbundle $F_0$ of rank $2$ of $T\P^3$ on $\P^3\setminus S(\F)$ defining $\F$ away from $S(\F)$.

The distribution $\F$ leaves the hyperplane $V=\{\z_3=0\}\simeq\P^2$ invariant and we work on $V$.
In fact  the distribution $\F$ also  leaves  the singular hypersurface $\{\zeta_0\zeta_3=0\}$, which contains the whole $S(\F)$, invariant. This case will be treated elsewhere \cite{Su9}.

Thus we consider the singular distribution $\F_V=\F\otimes {\cal O}_V$ on $V$, whose singular set $S$ is given by $S=S(\F)\cap V=S_1\cup S_2$. We let $P=(0 :1 :0 :0)$, which is the intersection point of $S_1$ and $S_2$.
 The restriction  of the bundle $F_{V,0}=F_0|_V$ defines $\F_V$ on $U_0=V\setminus S$. As is shown in Section \ref{CSaction},
the normal bundle $N_V$ of  $V$ in $\P^3$ admits a partial connection along $F_{V,0}$ on $U_0$ and the first Atiyah class $a^1(N_{V})$ is localized near $S$ and yield an ``Atiyah residue".


Note that, although the first Chern class $c^1(N_V)$ is not  a priori localized in this context, it has the ``Atiyah localization" and  the ``Atiyah residue", since it coincides with $a^1(N_{V})$,  $V$ being compact K\"ahler (see Remarks \ref{Chernnonvanish} and \ref{finalrmk} below).

To describe the localization more precisely, we need the \v Cech-Dolbeault cohomology theory for coverings involving more than two open sets, as $S$ is singular in our case.  We briefly recall what is needed in our case.


Let $U_0=V\setminus S$ be as above and let $U_1$, $U_2$ and $U_3$ be neighborhoods of $S_1\setminus\{P\}$, $S_2\setminus\{P\}$ and $P$ in $V$, respectively, such that
$U_1\subset W^{(0)}$, $U_2\subset W^{(2)}$ and $U_3\subset W^{(1)}$.
Then ${\cal U}=\{U_0,\dots,U_3\}$ is a covering of $V$ and ${\cal U}'=\{ U_1,U_2,U_3\}$ is a covering of $U'=U_1\cup U_2\cup U_3$, which is an open neighborhood of $S$ in $V$.
Letting $U_{ij}=U_i\cap U_j$ and $U_{ijk}=U_i\cap U_j\cap U_k$, we set
\begin{equationth}\label{CDcochains}
A^{p,q}(\U)=\oplus_{i} A^{p,q}(U_i)\oplus_{i,j} A^{p,q-1}(U_{ij})\oplus_{i,j,k} A^{p,q-2}(U_{ijk}),
\end{equationth}
where in the first sum, $0\le i\le 3$, in the second, $0\le i<j\le 3$ and in the third, $0\le i<j<k\le 3$.
The differential operator
\[
\bar D :A^{p,q}(\U)\lra A^{p,q+1}(\U)
\]
is defined by
\[
\bar D(\sigma_i,\sigma_{ij},\sigma_{ijk})=(\bar\partial\sigma_i,\sigma_j-\sigma_i-\bar\partial\sigma_{ij},\sigma_{jk}-\sigma_{ik}+\sigma_{ij}+\bar\partial\sigma_{ijk}).
\]

The $q$-th cohomology of the complex $(A^{p,*}(\U),\bar D)$ is the \v Cech-Dolbeault cohomology $H^{p,q}_{\bar D}({\cal U})$ of $\U$ of type $(p,q)$, which is shown to be canonically isomorphic to the Dolbeault cohomology $H^{p,q}_{\bar \partial}(V)$ of $V$ (cf. Theorem \ref{DCD}).

Likewise we have the cohomology $H^{p,q}_{\bar D}({\cal U}')$ of the complex $(A^{p,*}(\U'),\bar D)$ by omitting $U_0$ in the above.

Also, setting $A^{p,q}(\U,U_0)=\{\,\sigma\in A^{p,q}(\U)\mid\sigma_0=0\,\}$, we have the relative cohomology $H^{p,q}_{\bar D}({\cal U},U_0)$, which we also denote by $H^{p,q}_{\bar\partial}(V,V\setminus S)$.

The Atiyah classes are defined in the \v Cech-Dolbeault cohomology as in Subsection \ref{AtiyahinCD}, taking a $(1,0)$-connection on each open set and making use of difference forms.
In our case, the first Atiyah class $a^1(N_V)$ is represented by the cocycle $a^1(\na_*)$  in
\begin{equationth}\label{CD1.1}
A^{1,1}(\U)=\oplus_{i} A^{1,1}(U_i)\oplus_{i<j} A^{1,0}(U_{ij}),
\end{equationth}
(note that $A^{p,q-2}(U_{ijk})=0$ in (\ref{CDcochains}), if $(p,q)=(1,1)$) given by
\[
a^1(\na_*)=(a^1(\na_i),a^1(\na_i,\na_j)),
\]
with $\na_i$ a $(1,0)$-connection on $U_i$. If we take  an $F_{V,0}$-connection as $\na_0$,  we have $a^1(\na_0)=0$ (cf Theorem \ref{ABvanishing}).  Hence $a^1(\na_*)$   is in
$A^{1,1}(\U,U_0)$ and defines the localization $a^1(N_{V},\F_V)$ in $H^{1,1}_{\bar D}({\cal U},U_0)$.

Recall that $V$ is defined by $\z_3=0$ in $\P^3$. Thus, in $W^{(0)}$  it is defined by $z=0$ with $(x,y)$  coordinates
on $W^{(0)}\cap V (\supset U_1)$, in $W^{(2)}$ it is defined by $x_2=0$ with $(y_2,z_2)$  coordinates
on $W^{(2)}\cap V (\supset U_2)$ and in $W^{(1)}$  it is defined by $y_1=0$ with $(x_1,z_1)$  coordinates
on $W^{(1)}\cap V (\supset U_3)$.

\begin{proposition}\label{locexplicit} Let $\F$ be the singular distribution on $\P^3$ as above. It leaves the hyperplane $V$ given by $\z_3=0$ invariant.
We have the localization $a^1(N_{V},\F_V)$ in $H^{1,1}_{\bar D}({\cal U},U_0)$ of $a^1(N_{V})$ in $H^{1,1}_{\bar D}({\cal U})=H^{1,1}_{\bar\partial}(V)$. By a suitable choice of connections $\na_i$, it is represented by the \v Cech-Dolbeault cocycle $a^1(\na_*)=(a^1(\na_i),a^1(\na_i,\na_j))$ given by
\[
\begin{aligned}
&a^1(\na_i)=0, \ 0\le i\le 3,\qquad a^1(\na_0,\na_1)=\frac {\sqrt{-1}} {2\pi}\frac {dx+dy} y\\
&a^1(\na_0,\na_2)=\frac {\sqrt{-1}} {2\pi}\left(z_2\frac {dy_2} {y_2}-dz_2\right),\qquad a^1(\na_0,\na_3)=-\frac {\sqrt{-1}} {2\pi}\left(\frac {dx_1} {x_1z_1}-\frac {dz_1} {z_1}\right)\\
&a^1(\na_1,\na_2)=\frac {\sqrt{-1}} {2\pi}\frac {dy_2}{y_2},\qquad a^1(\na_1,\na_3)=\frac {\sqrt{-1}} {2\pi}\frac {dx_1}{x_1},\qquad a^1(\na_2,\na_3)=\frac {\sqrt{-1}} {2\pi}\frac {dz_1}{z_1}.
\end{aligned}
\]
\end{proposition}

\begin{proof} By taking  an $F_{V,0}$-connection for $N_V$ on $U_0$ as $\na_0$, we have $a^1(\na_i)=0$ as above.
We have the exact sequence
\[
0\lra TV\lra T\P^3|_V\overset{\pi}\lra N_{V}\lra 0.
\]

On each of $U_1$, $U_2$ and $U_3$, the bundle $N_V$ is trivial and we may take $\nu_1=\pi(\frac\partial {\partial z})$, $\nu_2=\pi(\frac\partial {\partial x_2})$ and
$\nu_3=\pi(\frac\partial {\partial y_1})$, respectively, as a frame of $N_V$. Let $\na_i$ be the connection trivial with respect to $\nu_i$. Then we have $a^1(\na_i)=0$, $1\le i\le 3$.

To compute the difference forms $a^1(\na_i,\na_j)$,  we first make the following observation (cf. Subsection \ref{subA}).
Let $\t_i$ be the connection matrix (form, in this case) of $\na_i$ with respect to some holomorphic frame $\nu$ of $N_V$. Then, since the $\t_i$'s are of type $(1,0)$,
\begin{equationth}\label{difference1}
a^1(\na_i,\na_j)=c^1(\na_i,\na_j)=\frac {\sqrt{-1}} {2\pi}(\t_j-\t_i).
\end{equationth}

Moreover, if $\tilde\nu=a\nu$ is another holomorphic frame and if the $\tilde\t_i$'s are corresponding connection forms, we have (cf. (\ref{framechange}))
\begin{equationth}\label{change1}
\tilde\t_i=\t_i+\frac {da} a.
\end{equationth}


We first compute $a^1(\na_0,\na_1)$. For this, we find the connection forms $\t_0$ and $\t_1$ of $\na_0$ and $\na_1$ with respect to the frame $\nu_1$. Since $\t_1=0$, we only need to find $\t_0$. Note that $U_{01}\subset W^{(0)}$, where we may take the vector fields $v_1$ and $v_2$ in (\ref{vfgen}) as generators of $\F$.
We set
\[
u_1=v_1|_V=y\frac\partial{\partial y}\quad\text{and}\quad u_2=v_2|_V=\frac \partial {\partial x}-\frac \partial {\partial y}.
\]

Since $\theta_0$ is  of type $(1,0)$, we may write as $\theta_0=f\,dx+g\,dy$.
Then, on the one hand we have $\na_0(\nu_1)(u_1)=yg\cdot\nu_1$ and $\na_0(\nu_1)(u_2)=(f-g)\cdot\nu_1$.
On the other hand by definition,
\[
\na_0(\nu_1)(u_1)=\pi\left(\left[y\frac\partial{\partial y}+z\frac\partial{\partial z},\frac\partial{\partial z}\right]{\vert_V}\right)=-\nu_1,
\]
and
\[
\na_0(\nu_1)(u_2)=\pi\left(\left[\frac\partial{\partial x}-\frac\partial{\partial y},\frac\partial{\partial z}\right]|_V\right)=0.
\]

Hence we get
\[
\t_0=-\frac {dx+dy} y,
\]
which gives the expression for $a^1(\na_0,\na_1)$ by (\ref{difference1}).





Similar computations show that the connection forms of $\na_0$ with respect to the frames $\nu_2$ and $\nu_3$ are, respectively,
$-z_2\frac {dy_2} {y_2}+dz_2$ and $\frac {dx_1} {x_1z_1}-\frac {dz_1} {z_1}$, which give
the expressions for $a^1(\na_0,\na_2)$ and $a^1(\na_0,\na_3)$.

Finally  the relations $\nu_2=\frac 1 {y_2}\nu_1$, $\nu_3=\frac 1 {x_1}\nu_1$ and $\nu_3=\frac 1 {z_1}\nu_2$ give the expressions for $a^1(\na_1,\na_2)$, $a^1(\na_1,\na_3)$ and $a^1(\na_2,\na_3)$ by (\ref{change1}).

\end{proof}

\begin{remark}\label{Chernnonvanish} From the above, we see that the curvature form of $\na_0$ with respect to $\nu_1$ is given by
\[
\kappa_0=d\t_0+\t_0\wedge\t_0=-\frac {dx\wedge dy} {y^2}.
\]
 Since it has no $(1,1)$-component, we verify $a^1(\na_0)=0$, while $c^1(\na_0)=\frac {\sqrt{-1}}{2\pi}\kappa_0$ does not vanish.
\end{remark}

We now try to find the corresponding residue.
For this, we first consider the cup product in our case. Recalling (\ref{CDcochains}) and (\ref{CD1.1}), it is a pairing
\[
A^{1,1}(\U)\times A^{1,1}(\U)\lra A^{2,2}(\U)
\]
given by
\[
(\sigma_i, \sigma_{ij},0)\smallsmile (\tau_i, \tau_{ij},0)=(\sigma_i\wedge\tau_i,\sigma_i\wedge\tau_{ij}+\sigma_{ij}\wedge\tau_j,-\sigma_{ij}\wedge\tau_{jk}).
\]

This induces a pairing $H^{1,1}_{\bar D}({\cal U})\times H^{1,1}_{\bar D}({\cal U})\lra H^{2,2}_{\bar D}({\cal U})$, which followed by integration $\int_V\colon H^{2,2}_{\bar D}({\cal U})\simeq H^{2,2}_{\bar \partial}(V)\lra\C$ defines
the Kodaira-Serre duality. 

In the relative case, we have $\sigma_0=0$ and the above cup product involves only $(\tau_i, \tau_{ij})$ with $i\ge 1$. Hence we have
the pairing
\[
A^{1,1}(\U,U_0)\times A^{1,1}({\cal U}')\lra A^{2,2}(\U,U_0).
\]
This in turn induces the pairing
\[
H^{1,1}_{\bar D}(\U,U_0)\times H^{1,1}_{\bar D}(\U')\lra H^{2,2}_{\bar D}(\U,U_0),
\]
which, followed by integration, defines the $\bar\partial$-Alexander homomorphism
\[
\bar A:H^{1,1}_{\bar D}(\U,U_0)\lra H^{1,1}_{\bar D}(\U')^*
\]
and we have a commutative diagram as in Proposition \ref{commutativity}, to which we come back  below (cf. (\ref{finalCD})).

We look the the $\bar\partial$-Alexander homomorphism more closely. We take a ``system of honeycomb cells" $(R_i)$ adapted to $\U$, which will be  given explicitly  below.
For a class $[\sigma]$ in $H^{1,1}_{\bar D}(\U,U_0)$, $\sigma=(\sigma_i, \sigma_{ij})$, the image of $[\sigma]$ by $\bar A$ is a functional assigning to each class $[\tau]$ in $H_{\bar D}^{1,1}({\cal U}')$, $\tau=(\tau_i, \tau_{ij})$, the integral

\begin{equationth}\label{resint}
\begin{aligned}
\int_V\sigma\smallsmile\tau=&\sum_{1\le i\le 3}\left(\int_{R_i}\sigma_i\wedge\tau_i+\int_{R_{0i}}\sigma_{0i}\wedge\tau_i\right)\\
&+\sum_{1\le i<j\le 3}\left(\int_{R_{ij}}\sigma_i\wedge\tau_{ij}+\sigma_{ij}\wedge\tau_j
-\int_{R_{0ij}}\sigma_{0i}\wedge\tau_{ij}\right).
\end{aligned}
\end{equationth}
In the above, each $R_i$ has the same orientation as $V$. We set $R_{ij}=R_i\cap R_j=\partial R_i\cap\partial R_j$, which has the same orientation as $\partial R_i$ (opposite orientation of $\partial R_j$) and $R_{0ij}=R_0\cap R_{ij}=\partial R_0\cap\partial R_{ij}$, which has the same orientation as $\partial R_{0i}$.

In fact, the right hand side of (\ref{resint}) can be reduced choosing Stein open sets as $U_i$, $1\le i\le 3$, which is possible (for example, we may take as $U_1$ a tubular neighborhood of $S_1\setminus \{P\}$ in $V\cap W^{(0)}$ containing $R_1$, or even the whole $V\cap W^{(0)}\simeq\C^2$).

\begin{lemma}\label{Stein} If we choose $U_i$, $1\le i\le 3$, to be Stein, we may represent every class in $H_{\bar D}^{1,1}({\cal U}')$ by a cocycle of the form $\xi=(0, \xi_{ij})$.
\end{lemma}

\begin{proof} From $\bar D\tau=0$, we have $\bar\partial\tau_i=0$, $1\le i\le 3$. Since each $U_i$ is Stein, there exist a $(1,0)$-form $\rho_i$ such that $\tau_i=\bar\partial\rho_i$.
If we set $\xi=(0,\xi_{ij})$ with
\[
\xi_{ij}=\tau_{ij}+\rho_i-\rho_j,
\]
Then we have $\tau=\xi+\bar D\rho$, $\rho=(\rho_i,0)$.
\end{proof}

If we use the representative as  above, the right hand side of (\ref{resint}) becomes
\begin{equationth}\label{funct}
\sum_{1\le i<j\le 3}\left(\int_{R_{ij}}\sigma_i\wedge\xi_{ij}
-\int_{R_{0ij}}\sigma_{0i}\wedge\xi_{ij}\right).
\end{equationth}

Recall that the residue ${\rm Res}_{a^1}(\F_V,N_V;U')$ of $\F_V$ with respect to $a^1$ for $N_V$ on $U'$ is the image of the localization
$a^1(N_V,\F_V)$.

\begin{proposition} If we choose  connections $\na_i$ as in Proposition \ref{locexplicit} and a representative $\xi$ of each class in $H^{1,1}(\U')$ as in Lemma \ref{Stein}, the residue ${\rm Res}_{a^1}(\F_V,N_V;U')$ is the functional assigning  to  $[\xi]$ the value
\[
-\sum_{1\le i<j\le 3}\int_{R_{0ij}}a^1(\na_0,\na_i)\wedge\xi_{ij}.
\]
\end{proposition}
\begin{proof} The proposition follows from $a^1(\na_i)=0$ and (\ref{funct}).
\end{proof}

The domains of integrations  $R_{0ij}$ can be given  explicitly, for example,  as follows.
Let $\delta$  be positive number with $\delta^2<1$, and set
\[
\begin{aligned}
R_3&=\{\,\z\in V\mid  \  |\z_0|^2+|\z_2|^2\le\delta^2\, |\z_1|^2\,\},\qquad
R_1=\{\,\z\in V\mid \   |\z_2|^2\le\delta^2\, |\z_0|^2\,\}\setminus {\rm Int}\,R_3,\\
R_2&=\{\,\z\in V\mid \   |\z_0|^2\le\delta^2\, |\z_2|^2\,\}\setminus {\rm Int}\, R_3,\qquad
R_0=U_0\setminus (\cup_{i=1}^3 {\rm Int}\,R_i\cup_{1\le i<j\le 3}{\rm Int}\,R_{ij}).
\end{aligned}
\]


From $\delta<1$, we see that $R_{12}=\emptyset$ and thus $R_{012}=\emptyset$.
We first express $R_{013}$ explicitly. As a set, it is given by
\[
|y|=\delta,\quad 1+|y|^2=\delta^2\,|x|^2\quad\text{and}\quad z=0.
\]
Setting $\delta'=\frac {\sqrt{1+\delta^2}}\delta$, we have
\[
R_{013}=\{\,(x,y)\mid |x|=\delta', |y|=\delta\,\},
\]
oriented so that ${\rm arg}\,x\wedge {\rm arg}\, y$ is negative.
Similarly we have
\[
R_{023}=\{\,(y_2,z_2)\mid |y_2|=\delta,\  |z_2|=\delta'\,\},
\]
which is oriented so that ${\rm arg}\,y_2\wedge {\rm arg}\, z_2$ is positive.

Now we consider the commutative diagram

 \begin{equationth}\label{finalCD}
  \begin{CD}
 H^{1,1}_{\bar D}(\U,U_0)@>j^*>>   H^{1,1}_{\bar D}(\U) \simeq  H^{1,1}_{\bar\partial}(V)= H^{1,1}_{\bar\partial}(\P^2)\simeq H^2(\P^2,\C)\\
@VV{\bar A_V}V       @VV{KS_V=P_V} V    \\
H^{1,1}_{\bar D}(\U')^*@>i_*>>H^{1,1}_{\bar D}(\U)^*\simeq H^{1,1}_{\bar\partial}(V)^*=H^{1,1}_{\bar\partial}(\P^2)^*\simeq H_2(\P^2,\C).
\end{CD}
\end{equationth}

The normal bundle $N_V$ of  $V$ in $\P^3$ is isomorphic to the hyperplane bundle $H$ on $V=\P^2$.  Since $\P^2$ is compact K\"ahler, we know that the first Atiyah class $a^1(N_{V})$ in $H^{1,1}_{\bar\partial}(V)=H^2(\P^2,\C)\simeq\C$ coincides with the first Chern class  $c^1(N_V)=c^1(H)$, the generator of the cohomology.

We  try to find $i_*{\rm Res}_{c^1}(\F,N_V;S) $ and verify the Residue Theorem \ref{residueth}. Recall that the isomorphism $H^{1,1}_{\bar\partial}(\P^2)\lra H^{1,1}_{\bar D}(\U) $ is induced by $\tau\mapsto (\tau_i,\tau_{ij})=(\tau,0)$. Note that $H^{1,1}_{\bar\partial}(\P^2)\simeq\C$, which is generated  by  the class of
\[
\tau_0=\frac{\sqrt{-1}}{2\pi}\partial\bar\partial\log\Vert\z\Vert^2
\]
(cf. {\it e.g.}, \cite{GH}).
For $\tau_0$ we may take, as $\rho_i$ in the proof of Lemma \ref{Stein}, the forms
\[
\rho_1=-\frac{\sqrt{-1}}{2\pi}\frac{\bar x\,dx+\bar y\,dy}{1+|x|^2+|y|^2},\
\rho_2=-\frac{\sqrt{-1}}{2\pi}\frac{\bar y_2\,dy_2+\bar z_2\,dz_2}{1+|y_2|^2+|z_2|^2},\
\rho_3=-\frac{\sqrt{-1}}{2\pi}\frac{\bar x_1\,dx_1+\bar z_1\,dz_1}{1+|x_1|^2+|z_1^2}
\]
and we compute
\[
\xi_{13}=\rho_1-\rho_3=-\frac{\sqrt{-1}}{2\pi}\,\frac{dx}{x},\quad
\xi_{23}=\rho_2-\rho_3=-\frac{\sqrt{-1}}{2\pi}\,\frac{dz_2}{z_2}.
\]

Thus, to the canonical generator $[\tau_0]$, the residue assigns the value
\[
\begin{aligned}
-\int_{R_{013}}a^1(\na_0,\na_1)&\wedge\xi_{13}-\int_{R_{023}}a^1(\na_0,\na_2)\wedge\xi_{23}\\
&=\left(\frac {\sqrt{-1}} {2\pi}\right)^2\left\{\int_{R_{013}}\left(\frac {dx+dy} {y}\right)\wedge\frac{dx}{x}+\int_{R_{023}}\left(z_2\frac {dy_2} {y_2}-dz_2\right)\wedge\frac{dz_2}{z_2}\right\}\\
&=-\left(\frac {\sqrt{-1}} {2\pi}\right)^2\int_{R_{013}}\frac {dx\wedge dy} {xy}\\&=1,
\end{aligned}
\]
as expected, since $R_{013}$ is given by  $|x|=\delta'$ and $|y|=\delta$, oriented so that ${\rm arg}\,x\wedge {\rm arg}\, y$ is negative.


\begin{remark}\label{finalrmk} Although the first Chern class $c^1(N_V)$ is not localized as a Chern class (cf. Remark \ref{Chernnonvanish}), it has the ``Atiyah localization" and the ``Atiyah residue".
\end{remark}

\bibliographystyle{plain}

\end{document}